\documentclass[10pt,a4paper,twoside,reqno]{amsart}

\usepackage{xcolor,soul,lipsum}
\usepackage{hyperref,url}
\usepackage{amssymb,latexsym}
\usepackage{amsmath,amsthm}
\usepackage[latin1]{inputenc}
\usepackage{enumerate}
\usepackage[all]{xy} \CompileMatrices
\usepackage{amscd, array}
\usepackage{comment}
\hypersetup{colorlinks=false,pdfborderstyle={/S/U/W 0}}
\usepackage{slashed} 
\usepackage{tikz}
\usetikzlibrary{matrix}
\usepackage{enumitem}
\usepackage{multirow}

\SelectTips{cm}{12} 
\theoremstyle{plain}
\newtheorem{theorem}{Theorem}[section]
\newtheorem{lemma}[theorem]{Lemma}

\newtheorem{proposition}[theorem]{Proposition}

\theoremstyle{definition}
\newtheorem{definition}[theorem]{Definition}
\newtheorem{definition-theorem}[theorem]{Definition-Theorem}
\newtheorem{example}[theorem]{Example}
\theoremstyle{remark}
\newtheorem{remark}[theorem]{Remark}

\numberwithin{equation}{section} 

\usepackage{geometry}
\usepackage{enumitem}

\newcommand{\End}{\operatorname{End}}

\newcommand{\Ker}{\operatorname{Ker}}

\newcommand{\tll}{\Theta_{ll}}
\newcommand{\tnn}{\Theta_{nn}}
\newcommand{\tln}{\Theta_{ln}}
\newcommand{\tun}{\Theta_{un}}
\newcommand{\tul}{\Theta_{ul}}
\newcommand{\tuu}{\Theta_{uu}}

\newcommand{\cX}{{\mathcal{X}}}

\newcommand{\surj}{\to\kern-1.8ex\to}

\newcommand{\cF}{\mathcal{F}}

\newcommand{\cB}{\mathcal{B}}
\newcommand{\Diff}{\mathrm{Diff}}

\newcommand{\Conf}{\mathrm{Conf}}
\newcommand{\Sol}{\mathrm{Sol}}

\newcommand{\Met}{\mathrm{Met}}

\newcommand{\fru}{\mathfrak{u}}
\newcommand{\frh}{\mathfrak{h}}
\newcommand{\frw}{\mathfrak{w}}

\newcommand{\fre}{\mathfrak{e}}
\newcommand{\G}{\mathrm{G}}

\newcommand{\mS}{\mathrm{S}}
\newcommand{\cL}{\mathcal{L}}
\newcommand{\Spin}{\mathrm{S}}
\newcommand{\Sl}{\mathrm{SL}}

\newcommand{\arxiv}[1]{{\tt
		\href{http://www.arXiv.org/abs/#1}{arXiv:#1}}}

\newcommand{\eqdef}{\stackrel{{\rm def.}}{=}}
\newcommand{\cI}{\mathcal{I}}
\newcommand{\dd}{\mathrm{d}}
\newcommand{\Cl}{\mathrm{Cl}}

\newcommand{\frf}{\mathfrak{f}}

\begin{document}

\title[Parallel spinors on globally hyperbolic Lorentzian four-manifolds]{Parallel spinors on globally hyperbolic Lorentzian four-manifolds}
\author[\'A. Murcia]{\'Angel Murcia}
\author[C. S. Shahbazi$^{\ast}$]{C. S. Shahbazi$^{\ast}$}

\address{Instituto de F\'isica Te\'orica UAM/CSIC, Madrid, Kingdom of Spain}
\email{angel.murcia@csic.es}

\address{Department of Mathematics, University of Hamburg, Germany}
\email{carlos.shahbazi@uni-hamburg.de}

\thanks{We would like to thank V. Cort\'es, T. Leistner and A. Moroianu for very interesting discussions and comments. Part of this work was undertaken during a visit of C.S.S. to the University Paris-Saclay under the Deutsch-Franz\"osische  Procope Mobilit\"at program. C.S.S. would like to thank this very welcoming institution for providing a nice and stimulating working environment. The work of \'A.M. is funded by the Spanish FPU Grant No. FPU17/04964, with additional support from the MCIU/AEI/FEDER UE grant PGC2018-095205-B-I00 and the Centro de Excelencia Severo Ochoa Program grant SEV-2016-0597. The work of C.S.S. is supported by the Germany Excellence Strategy \emph{Quantum Universe} - 390833306.}

\maketitle

\begin{abstract}
We investigate the differential geometry and topology of globally hyperbolic four-manifolds $(M,g)$ admitting a parallel real spinor $\varepsilon$. Using the theory of parabolic pairs recently introduced in \cite{Cortes:2019xmk}, we first formulate the parallelicity condition of $\varepsilon$ on $M$ as a system of partial differential equations, the \emph{parallel spinor flow equations}, for a family of polyforms on an appropriate Cauchy surface $\Sigma \hookrightarrow M$. Existence of a parallel spinor on $(M,g)$ induces a system of \emph{constraint} partial differential equations on $\Sigma$, which we prove to be equivalent to an exterior differential system involving a cohomological condition on the shape operator of the embedding $\Sigma\hookrightarrow M$. Solutions of this differential system are precisely the allowed initial data for the evolution problem of a parallel spinor and define the notion of \emph{parallel Cauchy pair} $(\fre,\Theta)$, where $\fre$ is a coframe and $\Theta$ is a symmetric two-tensor. We characterize all parallel Cauchy pairs on simply connected Cauchy surfaces, refining a result of Leistner and Lischewski. Furthermore, we classify all compact three-manifolds admitting parallel Cauchy pairs, proving that they are canonically equipped with a locally free action of $\mathbb{R}^2$ and are isomorphic to certain torus bundles over $S^1$, whose Riemannian structure we characterize in detail. Moreover, we classify all left-invariant parallel Cauchy pairs on simply connected Lie groups, specifying when they are allowed initial data for the Ricci flat equations and when the shape operator is Codazzi. Finally, we give a novel geometric interpretation of a class of parallel spinor flows and solve it in several examples, obtaining explicit families of four-dimensional Lorentzian manifolds carrying parallel spinors.
\end{abstract}

\setcounter{tocdepth}{1} 
\tableofcontents


\section{Introduction}
\label{sec:intro}


\noindent
Globally hyperbolic four-dimensional Lorentzian manifolds play a fundamental role in Lorentzian geometry and mathematical physics, especially in mathematical General Relativity, where they provide a natural class of four-dimensional space-times for which the initial value problem of Einstein field equations is well-posed \cite{Yvonne,Choquetbook}. A natural geometric condition to impose on a globally hyperbolic \emph{spin} four-manifold $(M,g)$ is the existence of a spinor parallel with respect to the Levi-Civita associated to $g$. Despite the fact that the local structure of Lorentzian four-manifolds admitting a parallel spinor is well-known since the early days of mathematical general relativity and supergravity \cite{Tod:1983pm}, see also \cite{Bryant}, the more refined global differential geometric and topological aspects of such Lorentzian manifolds have been addressed in the literature only recently \cite{BaumLeistnerLischewski,BaumMuller,LeistnerLischewski,LeistnerSchliebner,Lischewski:2015cya}, see also \cite{BarLarzLeistner,Bazaikin,Candela:2002rr,GlobkeLeistner,Schliebner} for related global problems in Lorentzian geometry. The global differential geometric and topological study of globally hyperbolic Lorentzian manifolds of special holonomy, of which manifolds admitting a parallel irreducible spinor constitute a particular class, was indeed proposed in \cite{LeistnerLischewski} as a long term research program in the study of Lorentzian manifolds of special geometric type. In the spirit of such proposal, the main goal of this article is to investigate the differential geometry and topology of connected, oriented and time-oriented globally hyperbolic Lorentzian four-manifolds $(M,g)$ carrying a real parallel spinor field $\varepsilon \in \Gamma(\mathrm{S}_g)$, understood as a section of a bundle of irreducible real Clifford modules over the bundle of Clifford algebras of $(M,g)$. In order to do this, we will exploit the theory of \emph{parabolic pairs}, recently developed in \cite{Cortes:2019xmk}, which provides an equivalent description of a parallel spinor as a pair of certain one-forms satisfying a specific system of first order partial differential equations. The theory of parabolic pairs is a particular case of a general framework developed in Op. Cit. to study irreducible real spinors satisfying a \emph{generalized} Killing spinor equation. 


The equivalent description of a parallel spinor as a parabolic pair allows us to give a novel formulation and characterization of the evolution problem and the corresponding constraint equations on a given Cauchy surface $\Sigma \hookrightarrow M$. In particular, we show that specifying allowed initial data on $\Sigma$ is equivalent to specifying a \emph{parallel} Cauchy pair. The latter is defined as a pair $(\fre,\Theta)$ where $\fre$ is a coframe on $\Sigma$ and $\Theta$ is a symmetric tensor of $(2,0)$ type satisfying a specific exterior differential system. If such pair satisfies also the constraint equations corresponding to the four-dimensional Ricci-flatness problem, we say that $(\fre,\Theta)$ is \emph{constrained Ricci-flat}. On the other hand, if $\Theta$ is a Codazzi tensor, we say that $(\fre,\Theta)$ is Codazzi. By the resolution of the initial value problem of a parallel spinor \cite{LeistnerLischewski,Lischewski:2015cya}, every parallel Cauchy pair $(\fre,\Theta)$ on $\Sigma$ admits a Lorentzian development equipped with a parallel spinor, a fact that we can exploit to study the existence of parallel spinors through the study of Cauchy pairs. Our first result in this direction is Theorem \ref{thm:universalcover} which refines, in the specific case of four Lorentzian dimensions, Theorem 4 in \cite{LeistnerLischewski}, of which we provide an alternative proof using the framework of parallel Cauchy pairs, see Remark \ref{remark:universalcover}. We recall that, on the other hand, the results of \cite{LeistnerLischewski} hold in every dimension.

Every parallel Cauchy pair $(\fre,\Theta)$ defines a canonical transversely orientable codimension one foliation $\cF_{\fre}\subset \Sigma$. Our next result is Theorem \ref{thm:CauchyPairsClosed}, which characterizes all Cauchy pairs and associated foliations $\cF_{\fre}$ on compact Cauchy surfaces $\Sigma$. Since the constraint equations of a parallel spinor correspond to a certain type of \emph{imaginary} generalized Killing spinor equations \cite{BaumMuller}, the previous theorems can be understood as classification results about three-manifolds admitting imaginary generalized Killing spinor equations. 

Parallel Cauchy pairs $(\fre,\Theta)$ admit a natural notion of \emph{left-invariance} when defined on three-dimensional Lie groups. We exploit this fact to provide a classification of all left-invariant Cauchy pairs on connected and simply connected Lie groups, specifying when they are in addition constrained Ricci-flat or Codazzi. This classification is presented in Theorem \ref{thm:allcauchygroups}, which shows that there exists a plethora of allowed left invariant initial data for the problem of a parallel spinor on a Lorentzian four-manifold, some of which satisfy also the constraint equations of Ricci-flatness or the Codazzi condition on $\Theta$, although not necessarily.

The evolution problem of a parallel Cauchy pair yields a complicated system of \emph{flow} equations, which we show to admit a simple geometric interpretation for \emph{comoving} globally hyperbolic Lorentzian manifolds in Theorem \ref{thm:comovingparallelflow}. The latter are defined as globally hyperbolic manifolds which admit a standard presentation in which $g(\partial_t,\partial_t) = - 1$, see Definitions \ref{def:comovingparallelspinorflow} and \ref{def:comovinggloballyhyper}. We exploit Theorem \ref{thm:comovingparallelflow} to construct large families of Lorentzian four-manifolds admitting parallel spinors, which are obtained as explicit solutions of the evolution equations of a Cauchy pair. We believe that the formulation of the evolution problem given in the previous theorem might be useful in order to obtain a simplified proof of the well-posedness of the initial value problem of a parallel spinor \cite{LeistnerLischewski,Lischewski:2015cya}. 


\subsection{Outline of the paper}


In Section \ref{sec:Spinors4d} we introduce the theory of parabolic pairs and parallel Cauchy pairs on globally hyperbolic Lorentzian four-manifolds, considering in addition the constrained Ricci-flat and Codazzi conditions. In Section \ref{sec:topologyCauchyPairs} we characterize parallel Cauchy pairs on simply connected Cauchy surfaces and we characterize all parallel Cauchy pairs and associated foliations in the compact case. In Section \ref{sec:leftinvariant} we classify all left-invariant parallel Cauchy pairs on simply connected Lie groups. Finally, in Section \ref{sec:examplesflows} we study a particular parallel spinor flow which we characterize geometrically and solve explicitly in particular cases. 


\subsection{Conventions}


We work in mostly plus signature, that is, Lorentzian metrics are always assumed to be of signature $(-,+,+,+)$, whence time-like vector fields have negative norm. Given a Lorentzian four-manifold $(M,g)$, every fiber of its bundle of Clifford algebras $\Cl(M,g)$ is isomorphic to the standard real Clifford algebra $\Cl(3,1)$ defined with the $+$ convention. That is, if $(e_0,e_1,e_2.e_3)$ is the standard orthonormal basis of four-dimensional Minkowski space $\mathbb{R}^{3,1}$, with $e_0$ time-like, then the following relations hold in $\Cl(3,1)$:
\begin{equation*}
e_0^2 = - 1\, , \quad e_1^2 = 1, \quad e_2^2 = 1, \quad e_3^2 = 1\, .
\end{equation*}

\noindent
Note that this convention is opposite to that of \cite{LawsonMichelsohn}.


\section{Parallel real spinors on Lorentzian four-manifolds}
\label{sec:Spinors4d}


In this section we develop the theory of parallel spinors on four-dimensional Lorentzian manifolds, assuming as the starting point of our investigation one of the main results of \cite{Cortes:2019xmk}, which characterizes parallel spinors in terms of a certain type of distribution satisfying a prescribed system of partial differential equations.


\subsection{General theory}
\label{subsec:GeneralSpinors4d}


Let $(M,g)$ be a four-dimensional \emph{space-time}, that is, a connected, oriented and time oriented Lorentzian four-manifold equipped with a Lorentzian metric $g$. We assume that $(M,g)$ is equipped with a bundle of irreducible real spinors $\mathrm{S}_g$. This is by definition a bundle of irreducible real Clifford modules over the bundle of Clifford algebras of $(M,g)$. Existence of such $\mathrm{S}_g$ is in general obstructed. The obstruction was shown in \cite{LazaroiuShahbazi,Lazaroiu:2017zpl} to be equivalent to the existence of a spin structure $Q_g$, in which case $\mS_g$ can be considered to be a vector bundle associated to $Q_g$ through the tautological representation induced by the natural embedding $\mathrm{Spin}_{+}(3,1)\subset \Cl(3,1)$, where $\mathrm{Spin}_{+}(3,1)$ denotes the connected component of the identity of the spin group in signature $(3,1) = -++\,+$ and $\Cl(3,1)$ denotes the real Clifford algebra in signature $(3,1)$. 

\begin{remark}
The \emph{tautological representation} of $\mathrm{Spin}_{+}(3,1) \subset \Cl(3,1)$ is the representation obtained by restriction of the unique irreducible real Clifford representation $\gamma\colon \Cl(3,1) \to \End(\mathbb{R}^4)$ of $\Cl(3,1)$. This representation is real of real type (the commutant of the image of $\gamma$ in $\End(\mathbb{R}^4)$ is trivial) and $\gamma$ is in fact an isomorphism of unital and associative algebras. In particular $\mathbb{R}^4$ admits a skew-symmetric non-degenerate bilinear pairing which is invariant under $\mathrm{Spin}_{+}(3,1)$ transformations \cite{ACDP,AC} (note that this bilinear \emph{cannot} be chosen to be symmetric). 
\end{remark}

\noindent
We will assume, without loss of generality, that $(M,g)$ is spin and equipped with a fixed spin structure $Q_g$. Then, the Levi-Civita connection $\nabla^g$ on $(M,g)$ induces canonically a connection on $\mS_g$, the \emph{spinorial} Levi-Civita connection, which we denote for simplicity by the same symbol. 

\begin{definition}
A \emph{spinor field} $\varepsilon$ on $(M,g,\mS_g)$ is a smooth section $\varepsilon\in \Gamma(\mS_g)$ of $\mS_g$. A spinor field $\varepsilon$ is said to be \emph{parallel} if $\nabla^g\varepsilon = 0$.
\end{definition}

\noindent
For every light-like one-form $u\in \Omega^1(M)$ we define an equivalence relation $\sim_u$ on the vector space of one-forms as follows. Given $l_1 , l_2 \in \Omega^1(M)$ the equivalence relation $\sim_u$ declares $l_1 \sim_u l_2$ to be equivalent if and only if $l_1 = l_2 + f u$ for a function $f\in C^{\infty}(M)$. We denote by:
\begin{equation*}
\Omega^1_{u}(M) \eqdef \frac{\Omega^1(M)}{\sim_u}\, ,
\end{equation*}

\noindent
the $C^\infty(M)$-module of equivalence classes defined by $\sim_u$.

\begin{definition}
A {\bf parabolic pair} $(u,[l])$ on $(M,g)$ consists of a nowhere vanishing null one-form $u\in \Omega^1(M)$ and an equivalence class of one-forms:
\begin{equation*}
[l]\in \Omega^1_u(M)\, ,
\end{equation*}

\noindent
such that the following equations hold:
\begin{equation*}
g(l,u) =  0\, , \qquad g(l,l) = 1\, ,
\end{equation*}

\noindent
for some, and hence for all, representatives $l\in [l]$.
\end{definition}

\noindent
The starting point of our analysis is the following result, which follows from \cite[Theorems 4.26 and 4.32]{Cortes:2019xmk} and gives the characterization of parallel spinors on $(M,g)$ that will be most convenient for our purposes.

\begin{proposition}
\label{prop:parallelspinorgeneral}
A space-time four-manifold $(M,g)$ admits a parallel spinor field $\varepsilon\in \Gamma(\mS_g)$ for some bundle of irreducible spinors $\mS_g$ over $(M,g)$ if and only if there exists a parabolic pair $(u,[l])$ on $(M,g)$ satisfying:
\begin{equation}
\label{eq:parallelspinorgeneral}
\nabla^g u = 0\, , \qquad \nabla^g l = \kappa\otimes u\, ,
\end{equation}

\noindent
for some representative (and hence for all) $l\in [l]$ and a one-form $\kappa \in \Omega^1(M)$.  
\end{proposition}

\begin{remark}
More precisely, Reference \cite{Cortes:2019xmk} proves that a nowhere vanishing spinor $\varepsilon \in \Gamma(\mathrm{S}_g)$ on $(M,g)$ defines a unique distribution of co-oriented parabolic two-planes in $M$, which in turn determines uniquely both $u$ and the equivalence class of one-forms $[l]$. Conversely, any such distribution determines a nowhere vanishing spinor on $(M,g)$, unique up to a global sign, with respect to a spin structure on $(M,g)$. Moreover, \cite[Theorem 4.26]{Cortes:2019xmk} establishes a correspondence between a certain type of first-order partial differential equations for $\varepsilon$ and their equivalent as systems of partial differential equations for $(u,[l])$, of which Equations \eqref{eq:parallelspinorgeneral} constitute the simplest case. The reader is referred to \cite{Cortes:2019xmk} for further details. 
\end{remark}

\begin{remark}
Given a parabolic pair $(u,[l])$, constructing its associated spinor field $\varepsilon \in \Gamma(\mS_g)$ can be difficult, since it requires computing the preimage of the polyform $u + u\wedge l$ through the \emph{square spinor map} \cite[\textsection IV]{LawsonMichelsohn}. This is however not problematic for our purposes, since we are not interested in the parallel spinor $\varepsilon\in \Gamma(\mS_g)$ \emph{per se} but only in the geometric and topological consequences of its existence. In this context, the main point of equations \eqref{eq:parallelspinorgeneral} and the general formalism presented in \cite{Cortes:2019xmk} is to provide a framework to study spinorial differential equations without having to consider the spinorial geometry of the underlying pseudo-Riemannian manifold $(M,g)$. This point of view is motivated by the study of supersymmetric solutions supergravity, where $\varepsilon$ corresponds to the \emph{supersymmetry parameter}, an auxiliary object that \emph{a priori} bears no physical meaning and is only used to define mathematically the notion of \emph{supersymmetric solution}. 
\end{remark}

\noindent
We will say that a parabolic pair $(u,[l])$ is \emph{parallel}  if it corresponds to a parallel spinor field, that is, if it satisfies Equations \eqref{eq:parallelspinorgeneral} for a representative $l\in [l]$. The dual $u^{\sharp}\in \mathfrak{X}(M)$ of $u$ is a parallel vector field on $M$ which is usually referred to as the \emph{Dirac current} of $\varepsilon$ in the literature. The fact that the Dirac current of $\varepsilon$ is always null is specific (although not exclusive) of the type of irreducible real representation $\gamma\colon \Cl(3,1)\to \End(\mathbb{R}^4)$ that we have used to construct the spinor bundle $\mS_g$. Indeed, it can be seen (see for instance \cite[Proposition 3.22]{Cortes:2019xmk}) that the pseudo-norm of the Dirac current $u^{\sharp}$ is given by the pseudo-norm of $\varepsilon$ computed with respect to the \emph{admissible bilinear pairing} $\cB$ used to construct $u^{\sharp}$. Admissible bilinear pairings were classified in \cite{ACDP,AC}, from which it follows that in our case there exist two admissible pairings, both of them skew-symmetric. Therefore, $\cB(\varepsilon,\varepsilon) = 0$ automatically and $u^{\sharp}$ is always null. It should be noted that the spinorial polyforms associated to the same spinor field $\varepsilon$ through the two different admissible bilinear pairings are related by Hodge duality.

Proposition \ref{prop:parallelspinorgeneral} immediately implies that four-dimensional space-times admitting a parallel spinor field whose Dirac current is complete are particular instances of \emph{Brinkmann manifolds}, which are precisely defined as space-times equipped with a complete and parallel null vector field \cite{Brinkmann}. Other well-known properties of space-times admitting a parallel spinor field, such as the special form of their Ricci tensor, are also immediate consequences of Proposition \ref{prop:parallelspinorgeneral}, which provides an adequate global and coordinate-independent framework to study the geometry and topology of four-dimensional space-times admitting parallel spinors. In particular, such framework seems to be specially well-adapted to prove \emph{splitting theorems} in the spirit of \cite{SilvaFlores}, where the global geometry of Brinkmann space-times was investigated.  

Recall that if a pair $(u,l)$, with $l \in [l]$, satisfies equations \eqref{eq:parallelspinorgeneral} with respect to a given $\kappa\in \Omega^1(M)$ then any other representative $l^{\prime} = l + f u$ satisfies again equation \eqref{eq:parallelspinorgeneral} with respect to the same null one-form $u$ and a possibly different one-form $\kappa^{\prime}$ given by:
\begin{equation*}
\kappa^{\prime} = \kappa + \dd f\, .
\end{equation*}

\noindent
Rather than investigating the global geometry and topology of general space-times admitting parallel spinors, exploiting for instance the refined screen bundle construction that can be developed in the presence of a parabolic pair, we restrict the \emph{causality} of $(M,g)$ and we assume in the following that $(M,g)$ is globally hyperbolic as proposed in \cite{BaumLeistnerLischewski,LeistnerLischewski}.
 

\subsection{Globally hyperbolic $(M,g)$}
\label{subsec:GloballyHyperbolic}


\noindent
Let $(M,g)$ be a globally hyperbolic four-dimensional space-time. A celebrated theorem of Bernal and S\'anchez \cite{Bernal:2003jb,Bernal:2004gm} states that in this case $(M,g)$ has the following isometry type:
\begin{equation}
\label{eq:globahyp}
(M,g) = (\mathbb{R}\times \Sigma, -\lambda^2_t \dd t\otimes \dd t + h_t)\, ,
\end{equation}

\noindent
where $t$ is the canonical coordinate on $\mathbb{R}$, $\left\{\lambda_t\right\}_{t\in\mathbb{R}}$ is a smooth family of nowhere vanishing functions on $\Sigma$ and $\left\{ h_t\right\}_{t\in\mathbb{R}}$ is a family of complete Riemannian metrics on $\Sigma$. From now on we consider the identification \eqref{eq:globahyp} to be fixed. We set:
\begin{equation*}
\Sigma_t \eqdef \left\{ t\right\}\times \Sigma \hookrightarrow M\, , \qquad \Sigma \eqdef \left\{ 0\right\}\times \Sigma \hookrightarrow M\, ,
\end{equation*}

\noindent
and define:
\begin{equation*}
\mathfrak{t}_t = \lambda_t\, \dd t\, ,
\end{equation*}

\noindent
to be the outward-pointing unit time-like one-form orthogonal to $\Sigma_t$ for every $t\in \mathbb{R}$. We will consider $\Sigma\hookrightarrow M$, endowed with the induced Riemannian metric:
\begin{equation*}
h \eqdef h_0\vert_{T\Sigma\times T\Sigma}\, ,
\end{equation*}

\noindent
to be the Cauchy hypersurface of $(M,g)$. The \emph{shape operator} or scalar second fundamental form $\Theta_t$ of the embedded manifold $\Sigma_t\hookrightarrow M$ is defined in the usual way as follows:
\begin{equation*}
\Theta_t  \eqdef \nabla^g \mathfrak{t}_t\vert_{T\Sigma_t\times T\Sigma_t}\, ,
\end{equation*}

\noindent
This definition  can be seen to be equivalent to:
\begin{equation*}
\Theta_t = - \frac{1}{2\lambda_t} \partial_t h_t \in \Gamma(T^{\ast}\Sigma_t\odot T^{\ast}\Sigma_t)\, .
\end{equation*}

\noindent
Moreover, it can be seen that:
\begin{equation*}
\nabla^g \alpha \vert_{T\Sigma_t\times TM} = \nabla^{h_t} \alpha + \Theta_t(\alpha)\otimes \mathfrak{t}_t\, , \qquad \forall\,\,\alpha\in \Omega^1(\Sigma_t)\, ,
\end{equation*}

\noindent
where $\nabla^{h_t}$ denotes the Levi-Civita connection on $(\Sigma_t,h_t)$ and $\Theta_t(\alpha) := \Theta_t(\alpha^{\sharp_{h_t}})$ is by definition the evaluation of $\Theta_t$ on the metric dual of $\alpha$. Given a parabolic pair $(u,[l])$, we write:
\begin{equation*}
u = u^0_t\, \mathfrak{t}_t + u^{\perp}_t\, , \qquad l = l^0_t\, \mathfrak{t}_t + l^{\perp}_t\in [l]\, ,
\end{equation*}

\noindent
where the superscript $\perp$ denotes orthogonal projection to $T^{\ast}\Sigma_t$ and where we have defined:
\begin{equation*}
u^0_t = - g(u,\mathfrak{t}_t)\, , \qquad l^0_t = - g(l,\mathfrak{t}_t)\, ,
\end{equation*}
 
\noindent
Using the previous orthogonal splitting of $u$ and $l$ we can obtain an equivalent characterization of parallel spinors on a globally hyperbolic space-time in terms of \emph{tensor flow equations} on $\Sigma$.

\begin{lemma}\cite[Lemma 3.1]{BaumLeistnerLischewski}
\label{lemma:projectionnablau}
Let $u\in \Omega^1(M)$ be a null one-form on the globally hyperbolic manifold \eqref{eq:globahyp}. Then, $\nabla^g u  = 0$ if and only if:
\begin{equation*}
(\nabla^g_{v_1} u)(v_2) = 0\, ,
\end{equation*}

\noindent
for every $v_1\in \mathfrak{X}(M)$ and every $v_2\in \mathfrak{X}(\Sigma_t)$.
\end{lemma}

\begin{proof}
We compute:
\begin{equation*}
0 = g(\nabla^g u , u) = u^0_t\, g(\nabla^g u , \mathfrak{t}_t) +  g(\nabla^g u , u^{\perp}_t) = u^0_t\, g(\nabla^g u , \mathfrak{t}_t)\, ,
\end{equation*}

\noindent
where we have used that the spatial projection of $\nabla^g u$ is zero by assumption.
\end{proof}

\begin{lemma} 
\label{lemma:globhyperspinor}
A globally hyperbolic four-manifold $(M,g)  = (\mathbb{R}\times \Sigma, -\lambda^2_t \dd t\otimes \dd t + h_t)$ admits a parabolic pair, and hence a parallel spinor field, if and only if there exists a family of orthogonal one-forms $\left\{ u^{\perp}_t, l^{\perp}_t \right\}_{t\in \mathbb{R}}$ on $\Sigma$ satisfying the following equations:
\begin{eqnarray}
\label{eq:globhyperspinorI}
& \partial_t u^{\perp}_t + \lambda_t \Theta_t(u^{\perp}_t) = u^0_t\dd \lambda_t\, , \qquad  u^0_t \partial_t l^{\perp}_t +\lambda_t  u^0_t \Theta_t(l^{\perp}_t) + \dd \lambda_t (l^{\perp}_t) u^{\perp}_t = 0\, ,   \\
\label{eq:globhyperspinorII}
& \nabla^{h_t} u^{\perp}_t + u^0_t \Theta_t = 0\, , \qquad u^0_t \nabla^{h_t} l^{\perp}_t = \Theta_t(l^{\perp}_t)\otimes u^{\perp}_t\, , 
\end{eqnarray}

\noindent
as well as:
\begin{equation}
\label{eq:restrictionul}
(u^0_t)^2 = \vert u^{\perp}_t \vert^2_{h_t}\, , \qquad \vert l^{\perp}_t\vert^2_{h_t} = 1\, ,
\end{equation}

\noindent
In particular, $\partial_t u^0_t = \dd \lambda_t(u^{\perp}_t)$ and $\dd u^0_t + \Theta(u^{\perp}_t) = 0$. If equations \eqref{eq:globhyperspinorI} and \eqref{eq:globhyperspinorII} are satisfied, the corresponding parabolic pair $(u,[l])$ is given by:
\begin{equation*}
u = u^0_t \mathfrak{t}_t +  u_t^{\perp}\, ,  \qquad  [l] = [ l_t^{\perp}]\, .
\end{equation*}

\noindent
where $\vert u_t^{\perp}\vert^2_{h_t} = h_t(u^{\perp}_t , u^{\perp}_t)$ and $\vert l_t^{\perp}\vert^2_{h_t} = h_t(l^{\perp}_t , l^{\perp}_t)$.
\end{lemma}

\begin{proof}
Let $(u,[l])$ be a parabolic pair satisfying equations \eqref{eq:parallelspinorgeneral}. Write $u = u^0_t\, \mathfrak{t}_t + u^{\perp}_t$. We can find a representative $l\in [l]$ such that:
\begin{equation*}
l =  l^{\perp}_t\in \Omega^1(\Sigma_t)\, , \qquad t\in \mathbb{R}\, ,
\end{equation*}

\noindent
that is, with $l$ purely spatial. Using this representative together with Lemma \ref{lemma:projectionnablau}, it follows that equations \eqref{eq:parallelspinorgeneral} are equivalent to:
\begin{equation*}
\nabla^g_{\partial_t} u\vert_{T\Sigma_t} = 0\, , \quad \nabla^g_{v_t} u\vert_{T\Sigma_t} = 0\, , \quad \nabla^g_{\partial_t} l^{\perp}_t = \kappa(\partial_t)\, u\, , \quad \nabla^g_{v_t} l^{\perp}_t = \kappa(v_t)\, u\, , \quad \forall \, \, v_t \in T\Sigma_t\, .
\end{equation*}

\noindent
Denote by $\kappa^{\perp}_t$ the spatial projection of $\kappa \in \Omega^1(M)$. We compute:
\begin{eqnarray*}
&\nabla^g_{\partial_t} u\vert_{T\Sigma_t} = \partial_t u^{\perp}_t + \lambda_t \Theta_t(u^{\perp}_t) - u^0_t\dd \lambda_t\, , \quad \nabla^g u\vert_{T\Sigma_t\times T\Sigma_t} = \nabla^{h_t} u^{\perp}_t + u^0_t \Theta_t\, , \\
&\nabla^g_{\partial_t} l^{\perp}_t =   \partial_t l^{\perp}_t - \dd \lambda_t (l^{\perp}_t) \mathfrak{t}_t +\lambda_t \Theta_t(l^{\perp}_t) = \kappa(\partial_t) (u^0_t \mathfrak{t}_t + u^{\perp}_t)\, ,\\
&\nabla^g l^{\perp}_t \vert_{T\Sigma_t\times TM} = \nabla^{h_t} l^{\perp}_t +\Theta_t(l^{\perp}_t)\otimes \mathfrak{t}_t  = \kappa^{\perp}_t\otimes  (u^0_t \mathfrak{t}_t + u^{\perp}_t)\, .
\end{eqnarray*}

\noindent
Isolating $\kappa$ in the previous equations we obtain:
\begin{equation*}
\kappa(\partial_t) = - \frac{1}{u^0_t} \dd \lambda_t(l^{\perp}_t)\, , \qquad \kappa^{\perp}_t = \frac{1}{u^0_t} \Theta_t(l^{\perp}_t)\, ,
\end{equation*}

\noindent
Plugging these equations back into the expressions for the covariant derivatives of $l^{\perp}_t$ we obtain all equations in \eqref{eq:globhyperspinorI} and \eqref{eq:globhyperspinorII}. The fact that these equations imply $\partial_t u^0_t = \dd \lambda_t(u^{\perp}_t)$ and $\dd u^0_t + \Theta(u^{\perp}_t) = 0$ follows now by respectively manipulating the time and exterior derivatives of $(u^0_t)^2 = \vert u^{\perp}_t \vert^2_{h_t}$. The converse follows directly by construction and hence we conclude.
\end{proof}

\noindent
Summarizing, equations \eqref{eq:globhyperspinorI}, \eqref{eq:globhyperspinorII} and \eqref{eq:restrictionul} contain the necessary and sufficient conditions for a four manifold $M$ to admit a parallel spinor field with respect to a globally hyperbolic metric in the form of flow equations on $\Sigma$, to which we will refer as the \emph{parallel spinor flow equations}. 

\begin{definition}
A \emph{parallel spinor flow} on a direct product manifold $M=\mathbb{R}\times \Sigma$, where $\Sigma$ is an oriented three-manifold, is a tuple:
\begin{equation*}
	(\left\{ \lambda_t \right\}_{t\in \mathbb{R}} , \left\{ h_t \right\}_{t\in \mathbb{R}},\left\{ u^0_t\right\}_{t\in \mathbb{R}},\left\{ u^{\perp}_t\right\}_{t\in \mathbb{R}} ,\left\{ l^{\perp}_t\right\}_{t\in \mathbb{R}})\, ,
\end{equation*}

\noindent
satisfying equations \eqref{eq:globhyperspinorI}, \eqref{eq:globhyperspinorII} and \eqref{eq:restrictionul}.
\end{definition}

\begin{remark}
In the previous discussion we have used informally the notion of \emph{family of tensors parametrized by} $\mathbb{R}$. This notion can be given a rigorous meaning as follows. A family of, say, one-forms $\left\{ \alpha_t\right\}_{t\in\mathbb{R}}$ on $\Sigma$ is by definition a smooth section $\alpha \colon \mathbb{R}\times \Sigma \to \mathrm{p}^{\ast}(T^{\ast}\Sigma)$ of the pull-back of $T^{\ast}\Sigma$ by the canonical projection $\mathrm{p}\colon \mathbb{R}\times \Sigma \to \Sigma$. Families of other types of tensors are defined similarly.
\end{remark}


\subsection{The constraint equations}


Using Lemma \ref{lemma:globhyperspinor} together with the resolution of the initial value problem for a parallel null spinor presented in \cite{BaumLeistnerLischewski,LeistnerLischewski,Lischewski:2015cya}, see also \cite{Ammann:2019xzb,BaumLeistnerBook}, we obtain the following characterization of parallel spinors on globally hyperbolic Lorentzian four-manifolds.

\begin{proposition}
\label{prop:spinoronrflat}
A globally hyperbolic four-manifold $(M,g)$ with Cauchy surface $\Sigma \hookrightarrow M$ and second fundamental form $\Theta\in \Gamma(T^{\ast}\Sigma\odot T^{\ast}\Sigma)$ admits a parallel spinor $\varepsilon\in\Gamma(\mathrm{S}_g)$ if and only if $\Sigma$ admits a pair of unit length orthonormal one-forms $(e_u,e_l)$ on $\Sigma$ satisfying the following equations:
\begin{equation}
\label{eq:spinoronSigma}
\nabla^h e_u + \Theta = \Theta(e_u)\otimes e_u \, , \quad   \nabla^h e_l= \Theta(e_l)\otimes e_u\, , \quad [\Theta(e_u)] = 0 \in H^{1}(\Sigma,\mathbb{R})\, ,
\end{equation}

\noindent
where $h$ is the complete Riemannian metric induced by $g$ on $\Sigma$. Furthermore, $(M,g)$ is Ricci-flat only if $h$ satisfies in addition the following equations:
\begin{equation}
\label{eq:RicciflatonSigma}
\mathrm{R}_h = \vert\Theta\vert^2_h - \mathrm{Tr}_h(\Theta)^2\, ,\qquad  \dd \mathrm{Tr}_h(\Theta) = \mathrm{div}_h(\Theta)\, ,
\end{equation}

\noindent
on $\Sigma$.
\end{proposition}

\begin{proof}
The constraint equations for the evolution problem posed by equations \eqref{eq:parallelspinorgeneral} on the globally hyperbolic four-manifold $(M,g) = (\mathbb{R}\times \Sigma , -\lambda^2_t \dd t\otimes \dd t + h_t)$ are obtained by restriction of equations \eqref{eq:globhyperspinorII} together with equations \eqref{eq:restrictionul} to the Cauchy surface $\Sigma \hookrightarrow M$. The restriction is given by:
\begin{equation}
\label{eq:cauchyul}
(u^0)^2 = \vert u^{\perp} \vert^2_{h}\, , \qquad \vert l^{\perp}\vert^2_{h} = 1\, , \qquad \nabla^h u^{\perp} + u^0 \Theta = 0\, , \qquad u^0\,\nabla^{h}l^{\perp} = \Theta(l^{\perp})\otimes u^{\perp}\, ,
\end{equation}

\noindent
where we have set:
\begin{equation*}
u^0 \eqdef u^0_0\, , \qquad u^{\perp}\eqdef u^{\perp}_0\, , \qquad l^{\perp} \eqdef l^{\perp}_0\, , \qquad \Theta \eqdef \Theta_0\, .
\end{equation*}

\noindent
Defining now:
\begin{equation*}
e_u \eqdef \frac{u^{\perp}}{u^0}\, ,\qquad e_l \eqdef l^{\perp}\, , 
\end{equation*}

\noindent
the third and fourth equations in \eqref{eq:cauchyul} are equivalently to:
\begin{equation*}
\label{eq:cauchyul2}
\nabla^h e_u + \dd \log(u^0)\otimes e_u + \Theta = 0\, , \qquad  \nabla^{h}l^{\perp} = \Theta(l^{\perp})\otimes e_u\, .
\end{equation*}

\noindent
On the other hand, taking the exterior derivative of the first equation in \eqref{eq:cauchyul} we obtain:
\begin{equation*}
\dd \log(u^0) + \Theta(e_u) = 0\, ,
\end{equation*}

\noindent
which implies the third equation in \eqref{eq:spinoronSigma}. Combining this equation with equations \eqref{eq:cauchyul2} we obtain the first two equations in \eqref{eq:spinoronSigma}. Clearly we have $\vert e_u\vert_h^2 = \vert e_l\vert_h^2 = 1$ by construction. Conversely, assume that a pair of orthonormal one-forms $(e_u,e_l)$ satisfies equations \eqref{eq:spinoronSigma} for a given Riemannian metric $h$ and tensor $\Theta$. Write:
\begin{equation}
\label{eq:exactthetaeu}
d\mathfrak{f} = - \Theta(e_u)\, ,
\end{equation}

\noindent
for a function $\mathfrak{f}\in C^{\infty}(\Sigma)$, which exists since $[\Theta(e_u)] = 0\in H^1(\Sigma,\mathbb{R})$. Then, the triple:
\begin{equation*}
u^0 = e^{\mathfrak{f}}\, , \qquad u^{\perp} = e^{\mathfrak{f}} e_u\, , \qquad l^{\perp} = e_l\, ,
\end{equation*}

\noindent
is by construction a solution equations \eqref{eq:cauchyul}. Equation \eqref{eq:exactthetaeu} determines $u^0$ modulo constant rescalings, in agreement with the fact that if $(u,[l])$ is a parallel parabolic pair then so is $(c\,u,[l])$ for every $c\in \mathbb{R}^{\ast}$. Conversely, since the initial value problem of a parallel null spinor is well-posed by the results of \cite{BaumLeistnerLischewski,LeistnerLischewski,Lischewski:2015cya}, and a parallel spinor is equivalent to a parallel parabolic pair (see Proposition \ref{prop:parallelspinorgeneral}), every solution to \eqref{eq:spinoronSigma} admits a Lorentzian development carrying a parallel spinor and containing as Cauchy surface the submanifold $(\Sigma,h)$ with associated second fundamental form $\Theta$. The statement regarding the Ricci-flat condition follows from the celebrated resolution of the initial value problem of a Ricci-flat Lorentzian four-manifold, see \cite{YvonneGeroch,Yvonne}.
\end{proof}

\begin{remark}
The constraint equations corresponding to a parallel spinor on a globally hyperbolic Lorentzian manifold are well known to correspond to the \emph{imaginary} generalized Killing spinor equation with respect to the shape operator of the Cauchy hypersurface \cite{BarGauduchonMoroianu,BaumMuller,LeistnerLischewski}. Such type of characterization also applies to our problem, however we do not need to consider it thanks to the description of parallel spinors as parabolic pairs provided in Proposition \ref{prop:parallelspinorgeneral}.  
\end{remark}

\noindent
In Reference \cite{BaumMuller} the authors study imaginary \emph{Codazzi spinors}, which correspond to the constraint equations of a parallel spinor on a globally hyperbolic Lorentzian manifold of constant curvature. More recently, Reference \cite{LeistnerLischewski} determines the local isometry type of the Cauchy surface of any Lorentzian manifold carrying a parallel spinor, showing that, in the four-dimensional case, corresponds to a certain warped product involving a family of two-dimensional flat metrics. Therefore, the results of this article can be considered as a continuation of those in Op. Cit. in the specific case of four Lorentzian dimensions. The system of equations \eqref{eq:RicciflatonSigma} corresponds with the celebrated \emph{constraint equations} of the initial value problem for Ricci-flat globally hyperbolic Lorentzian four manifolds and consequently has been intensively and extensively studied in the literature. The first equation in \eqref{eq:RicciflatonSigma} is usually called the \emph{Hamiltonian constraint} whereas the second equation in \eqref{eq:RicciflatonSigma} is usually called the \emph{momentum constraint}. 

\begin{remark}
Equations \eqref{eq:spinoronSigma} contain a \emph{cohomological} condition, namely $[\Theta(e_u)] = 0$ which is automatically satisfied if $H^1(\Sigma,\mathbb{R}) = 0$. However, it may restrict the discrete quotients to which a given solution descends, since an exact one-form on $\Sigma$ may descend to a closed non-exact one-form on certain quotients of $\Sigma$. 
\end{remark}


\subsection{Parallel Cauchy pairs on $\Sigma$}


The variables of equations \eqref{eq:spinoronSigma} corresponding to the restriction of a parallel spinor field to the Cauchy surface consist of a pair of orthonormal one-forms $(e_u,e_l)$ on $(\Sigma,h)$. However, in order to study the geometry and topology of Cauchy surfaces on Lorentzian four-manifolds equipped with a parallel spinor it is convienent to consider also the Riemannian metric $h$ and the symmetric $(2,0)$ tensor $\Theta$ as variables of \eqref{eq:spinoronSigma}. We will refer to a symmetric tensor $\Theta \in \Gamma(\mathrm{S}^2 T^{\ast}\Sigma)$ on $\Sigma$ simply as a \emph{shape operator}. Following standard usage in the literature, if the shape operator of a given solution $(h,\Theta,e_u,e_l)$ satisfies:
\begin{equation*}
\nabla^h\Theta \in \Gamma(\mathrm{S}^3 T^{\ast}\Sigma)\, ,
\end{equation*}

\noindent
we will say that $\Theta$ is a \emph{Codazzi tensor} on $\Sigma$. More explicitly, a shape operator $\Theta$ is a Codazzi tensor if and only if:
\begin{equation*}
(\nabla^h_{v_1}\Theta)(v_2 , v_3) = (\nabla^h_{v_2}\Theta)(v_1 , v_3)\, ,
\end{equation*}

\noindent
for every $v_1 , v_2 , v_3 \in \mathfrak{X}(\Sigma)$. Denote by $\mathrm{F}(\Sigma)$ the principal bundle of oriented coframes of $\Sigma$. In order to proceed further we will first rewrite equations \eqref{eq:spinoronSigma} in a more transparent geometric form.

\begin{lemma}
There is a canonical one to one correspondence between tuples $(h,\Theta,e_u,e_l)$ as described above and pairs $( \mathfrak{e}, \Theta)$, where $\mathfrak{e}\colon \Sigma \to \mathrm{F}(\Sigma)$ is a section of $\mathrm{F}(\Sigma)$ and $\Theta$ is a shape operator.
\end{lemma}

\begin{proof}
Given $(h,\Theta,e_u,e_l)$, set:
\begin{equation*}
\fre = (e_u,e_l,e_n \eqdef \ast_h (e_u\wedge e_l))\, ,
\end{equation*}

\noindent
which is clearly a section of $\mathrm{F}(\Sigma)$. Conversely, given pair $(\fre,\Theta)$, write:
\begin{equation*}
\fre = (e_u, e_l ,e_n)\, ,
\end{equation*}

\noindent
and map $(\fre,\Theta)$ to the tuple $(h,\Theta,e_u,e_l)$, where:
\begin{equation*}
h = e_u\otimes e_u + e_l\otimes e_l + e_n\otimes e_n\, .
\end{equation*}

\noindent
Such tuple is mapped back again to $(\mathfrak{e} = (e_u, e_l ,e_n),\Theta)$ by the previous correspondence and hence we obtain the desired one to one map.
\end{proof}

\noindent
Therefore, in the following we will consider coframes and shape operators on $\Sigma$ as variables of equations \eqref{eq:spinoronSigma}.  

\begin{proposition}
\label{prop:exder}
Equations \eqref{eq:spinoronSigma} are equivalent to the following system of first-order partial differential equations: 
\begin{equation}
\label{eq:exder}
\dd \fre=\Theta(\fre) \wedge e_u \, , \qquad  [\Theta(e_u)]=0\, .
\end{equation}

\noindent
for pairs $(\fre = (e_u,e_l,e_n),\Theta)$.
\end{proposition}

\begin{remark}
More explicitly, equation $\dd \fre=\Theta(\fre) \wedge e_u$ corresponds to the following conditions:
\begin{equation*}
\dd e_u = \Theta(e_u) \wedge e_u\, , \qquad \dd e_l = \Theta(e_l) \wedge e_u\, , \qquad \dd e_n=\Theta(e_n) \wedge e_u\, ,
\end{equation*}

\noindent
where $\fre = (e_u,e_l,e_n)$. It should be noted that generalized Killing spinors on Riemannian three-manifolds, which differ from their \emph{imaginary} version which is obtained in the Lorentzian framework considered in this article, can also be studied in terms of a global coframe satisfying a given exterior differential system, see \cite{MoroianuSemmelmann} for more details.
\end{remark}

\begin{proof}
Suppose that $(\fre , \Theta)$ is a solution of equations \eqref{eq:spinoronSigma} where $h = e_u\otimes e_u + e_l\otimes e_l + e_n\otimes e_n$ and:
\begin{equation*}
\fre = (e_u,e_l,e_n = \ast_h (e_u\wedge e_l))\, .
\end{equation*}

\noindent
A direct computation shows that:
\begin{equation*}
\nabla^h e_n = \nabla^h  \ast_h (e_u\wedge e_l) =    \ast_h (\nabla^h e_u\wedge e_l) + \ast_h ( e_u\wedge \nabla^h e_n) = \ast_h (\nabla^g e_u\wedge e_l) = \Theta(e_n) \otimes e_u \, .
\end{equation*}

\noindent
The skew-symmetrization of the previous equation together with the  skew-symmetrization of the first two equations in \eqref{eq:spinoronSigma} yields, together with the cohomological condition, equations \eqref{eq:exder}. The converse follows easily by interpreting the first equation in \eqref{eq:exder} as the first Cartan structure equations for the coframe $\fre$, considered as orthonormal with respect to the metric $h = e_u\otimes e_u + e_l\otimes e_l + e_n\otimes e_n$. 
\end{proof}

\begin{remark}
Leaving aside the cohomological condition, Equations \eqref{eq:spinoronSigma} form a set of (a priori) nine independent equations. This is exactly the same number of (a priori) independent equations occurring in \eqref{eq:exder}, reflecting thus the equivalence between Equations \eqref{eq:spinoronSigma} and \eqref{eq:exder}. 
\end{remark}

\noindent
We will refer to equations \eqref{eq:exder} as the \emph{parallel Cauchy differential system}, which yields the constraint equations of a parallel spinor field on a globally hyperbolic four-dimensional space-time and will be the main object of study in this article.  

\begin{definition}
A \emph{Cauchy pair} $(\fre,\Theta)$ consists of a coframe $\fre$ and a symmetric $(2,0)$ tensor $\Theta$. A \emph{parallel Cauchy coframe} with respect to $\Theta$ is a coframe $\fre$ on $\Sigma$ such that $(\fre,\Theta)$ satisifes the parallel Cauchy differential system \eqref{eq:exder}. A \emph{parallel Cauchy pair} $(\fre,\Theta)$ is a Cauchy pair satisfying the parallel Cauchy differential system \eqref{eq:exder}.
\end{definition}

\noindent
The Riemannian metric associated to a parallel Cauchy coframe $\fre = (e_u,e_l,e_n)$, with respect to which equations \eqref{eq:spinoronSigma} are satisfied is defined as follows:
\begin{equation*}
h_{\fre} \eqdef e_u\otimes e_u + e_l\otimes e_l + e_n\otimes e_n\, .
\end{equation*}

\begin{remark}
As explained above, a Cauchy pair $(\fre,\Theta)$ defines a Riemannian metric $h_{\fre}$. Therefore, it is natural to impose the constraint equations of the four-dimensional Ricci-flatness problem on a pair $(h_{\fre},\Theta)$ associated to a parallel Cauchy pair $(\fre,\Theta)$. Such data $(\fre,\Theta,h_{\fre})$ would satisfy both the constraint equations of the parallel spinor and Ricci-flat problems. However, and to the best knowledge of the authors, this is not enough to guarantee that $(\Sigma,h_{\fre})$ admits a Lorentzian development which at the same time is Ricci flat and admits a parallel spinor. The reason is that the evolution of the initial data prescribed by the parallel spinor and Ricci flat problems need not be isomorphic. 
\end{remark}

\noindent
We provide now an example of a parallel Cauchy pair on a non-flat Riemannian three-manifold.

\begin{example}
\label{ep:tau3mu}
Take $\Sigma = \tau_{3,\mu}$ to be the simply-connected non-unimodular Lie group $\tau_{3,\mu}$ where $-1 < \mu \leq 1$, $\mu \neq 0$, is a constant, see \cite[Chapter 7]{Gorbatsevich} for its precise definition. On $\tau_{3,\mu}$ there exists a left-invariant co-frame $(e^1,e^2,e^3)$ satisfying:
\begin{equation*}
\dd e^1 = 0\, , \qquad \dd e^2 = \mu \, e^2\wedge e^1\, , \qquad \dd e^3 = e^3 \wedge e^1\, .
\end{equation*}
	
\noindent
Set:
\begin{equation*}
\fre = (e_u, e_l, e_n) := (e^1,e^2,e^3) \, , \quad h_{\fre} = e_u\otimes e_u + e_l\otimes e_l + e_n\otimes e_n\, , \quad \Theta := h + (\mu - 1)\, e_l\otimes e_l\, .
\end{equation*}
	
\noindent
A direct computation shows that $(\fre,\Theta)$ defines a parallel Cauchy pair on $\tau_{3,\mu}$, that is, $(\fre,\Theta)$ is a solution of equations \eqref{eq:exder}, or, equivalently, equations \eqref{eq:spinoronSigma}. Note that since $\dd e_u = 0$ and $\tau_{3,\mu}$ is simply connected, the one-form $e_u = \Theta(e_u)$ is automatically exact. In particular, we have:
\begin{equation*}
\nabla^h e_u = -\mu\, e_l\otimes e_l - e_n\otimes e_n\, , \qquad \nabla^h e_l = \mu \, e_l\otimes e_u\, , \qquad \nabla^h e_n = e_n\otimes e_u\, ,
\end{equation*}
	
\noindent
conditions which are equivalent to equations \eqref{eq:spinoronSigma}. More explicitly, write $e_u = \dd \mathfrak{f}$ for a real function $\mathfrak{f}\in C^{\infty}(\Sigma)$. Then $(\hat{e}_l = e^{\mu\mathfrak{f}} e_l , \hat{e}_n = e^{\mathfrak{f}} e_n)$ defines a pair of closed nowhere vanishing one-forms. In particular, $\hat{\fre} = (e_u,\hat{e}_l , \hat{e}_n)$ is a closed global coframe on $\Sigma$. Set:
\begin{equation*}
h_{\hat{\fre}} \eqdef e_u\otimes e_u + \hat{e}_l\otimes \hat{e}_l + \hat{e}_n\otimes \hat{e}_n\, ,
\end{equation*}
	
\noindent
to be the Riemannian metric defined by $\hat{\fre} \eqdef (e_u, \hat{e}_l, \hat{e}_n)$. Since $\dd\hat{\fre} = 0$, the metric $h_{\hat{\fre}}$ is flat and therefore:
\begin{equation*}
h_{\fre} = e_u\otimes e_u + e^{-2\mu\mathfrak{f}} \hat{e}_l\otimes \hat{e}_l + e^{-2\mathfrak{f}} \hat{e}_n\otimes \hat{e}_n\, ,
\end{equation*}
	
\noindent
is a warped product of flat metrics. Even more, since $\hat{\fre} = (e_u,\hat{e}_l , \hat{e}_n)$ is a closed coframe there exist local coordinates $(z,x,y)$ (global, if $\hat{\fre}$ is complete) such that:
\begin{equation*}
e_u = \dd \mathfrak{f} = \dd z\, , \qquad \hat{e}_l = \dd x\, , \qquad \hat{e}_n = \dd y\, .
\end{equation*}
	
\noindent
Therefore, the metric can be written as follows:
\begin{equation*}
h_{\fre} = \dd z\otimes \dd z + e^{- 2 \mu z} \dd x\otimes \dd x + e^{- 2 z}\dd y\otimes \dd y\, .
\end{equation*}
	
\noindent
The scalar curvature of $h_{\fre}$ can be computed to be:
\begin{equation*}
\mathrm{R}_h = -2 (1 + \mu + \mu^2)\, .
\end{equation*}
	
\noindent
which, together with the fact that $\vert\Theta\vert^2_h = 2+\mu^2$ and $\mathrm{Tr}_h(\Theta)^2 = (2+\mu)^2$ shows that the Hamiltonian constraint is satisfied if and only if\footnote{Recall that $-1 < \mu \leq 1$ and $\mu\neq 0$.}:
\begin{equation*}
\mu  = 1  \, .
\end{equation*}
	
\noindent
Since the momentum constraint is clearly satisfied if and only if $\mu =1$, we conclude that if $\mu\neq 1$ we obtain a solution to the constraint equations \eqref{eq:globhyperspinorI} and \eqref{eq:globhyperspinorII} whose Lorentzian development yields a non Ricci-flat Lorentzian four manifold. On the other hand, if $\mu =1$, the Riemannian three-manifold $(\Sigma,h_{\fre})$ admits Lorentzian developments (not necessarily equal) which are either Ricci flat, admit a parallel spinor or both. In all these cases (with $\mu=1$), $(\Sigma,h_{\fre})\hookrightarrow (M,g)$ is a totally umbilical submanifold of $(M,g)$.  
\end{example}

\noindent
A Cauchy pair is said to be \emph{complete} if $(\Sigma,h_{\fre})$ is a complete Riemannian three-manifold. When necessary, the dual of a Cauchy coframe $\fre$ will be denoted by $\fre^{\sharp} = (e^{\sharp}_u,e^{\sharp}_l,e^{\sharp}_n)$. Denote by:
\begin{equation*}
\Conf(\Sigma) \eqdef \Gamma(\mathrm{S}^2 T^{\ast}\Sigma) \times \Gamma(\mathrm{F}(\Sigma))\, ,
\end{equation*}

\noindent
the configuration space of the Cauchy differential system that is, its space of variables. Likewise, denote by:
\begin{equation*}
\Sol(\Sigma)\subset \Conf(\Sigma)\, , 
\end{equation*}

\noindent
the subspace of solutions of the Cauchy differential system. We have a canonical map:
\begin{equation*}
\Sol(\Sigma) \to \Met_c(\Sigma)\, , \qquad (\fre,\Theta) \mapsto h_{\fre}\, ,
\end{equation*}

\noindent
where $\Met_c(\Sigma)$ denotes the space of complete Riemannian metrics on $\Sigma$. The image of the previous map, which we denote by $\Met^s_c(\Sigma)$, is by definition the space of complete Riemannian metrics on $\Sigma$ that admit a solution to the Cauchy differential system for a shape operator $\Theta \in  \Gamma(\mathrm{S}^2 T^{\ast} \Sigma)$. The group of orientation preserving diffeomorphisms $\Diff(\Sigma)$ has a natural left action on $\Conf(\Sigma)$ given by push-forward:
\begin{equation*}
\mathbb{A}\colon \Diff(\Sigma)\times \Conf(\Sigma) \to \Conf(\Sigma)\, , \quad (\fru,(\fre,\Theta))\mapsto (\fru_{\ast}\fre,\fru_{\ast}\Theta)\, .
\end{equation*}

\noindent
For every $\fru\in \Diff(\Sigma)$, define:
\begin{equation*}
\mathbb{A}_{\fru}\colon \Conf(\Sigma) \to \Conf(\Sigma)\, , \quad (\fre,\Theta)\mapsto (\fru_{\ast}\fre,\fru_{\ast}\Theta)\, .
\end{equation*}

\begin{lemma}
\label{lemma:equivariancia}
Let $\fru \in \Diff(\Sigma)$. Then, $(\fre ,\Theta) \in \Sol(\Sigma)$ if and only if $\mathbb{A}_{\fru}(\fre ,\Theta) \in \Sol(\Sigma)$.
\end{lemma}

\begin{proof}
We compute:
\begin{equation*}
\dd e_u=\Theta(e_u) \wedge e_u  \Leftrightarrow \fru_{\ast}\dd e_u= \fru_{\ast}(\Theta(e_u) \wedge e_u) \Leftrightarrow \dd \fru_{\ast}e_u= (\fru_{\ast}\Theta)(\fru_{\ast}e_u) \wedge \fru_{\ast}e_u\, ,
\end{equation*}
	
\noindent
and similarly for the remaining equations of the Cauchy differential system \eqref{eq:exder}.
\end{proof}

\noindent
Therefore, the orientation-preserving diffeomorphism group of $\Sigma$ has a well-defined action on the space of parallel Cauchy pairs and we can consider the quotient:
\begin{equation*}
\mathfrak{M}(\Sigma) \eqdef \Sol(\Sigma)/\Diff(\Sigma)\, ,
\end{equation*}

\noindent
defined by the action $\mathbb{A}$. We call $\mathfrak{M}(\Sigma)$ the \emph{moduli space} of parallel Cauchy pairs on $\Sigma$, which we plan to investigate in a separate publication. In the following we will consider two parallel Cauchy pairs to be isomorphic if they are related by an orientation preserving diffeomorphism of $\Sigma$ as prescribed by the action $\mathbb{A}$. 


\subsection{The momentum and Hamiltonian constraints}
\label{subsec:ConstraintRicciflat}

 
In this section we consider the interplay between the Cauchy differential system and the constraint equations on $\Sigma$ induced from imposing Ricci-flatness in four dimensions. Recall that, in contrast with the situation occurring in Riemannian geometry, not every Lorentzian space-time admitting a parallel spinor is necessarily Ricci-flat, see for instance \cite{Bryant} for more details and explicit examples.

\begin{definition}
A parallel Cauchy pair $(\fre,\Theta)$ is \emph{constrained Ricci-flat} if it satisfies Equations \eqref{eq:RicciflatonSigma}, that is, if $h_{\fre}$ satisfies the constraint equations corresponding to the initial value problem posed by the Ricci-flatness condition in four dimensions.
\end{definition}

\begin{lemma}
\label{lemma:eumc}
Let $(\fre , \Theta)$ be a parallel Cauchy pair and write $\fre = (e_u,e_l,e_n)$. Then: 
\begin{equation*}
\mathrm{div}_h (\Theta)\wedge e_u = \dd \mathrm{Tr}(\Theta) \wedge e_u\, .
\end{equation*}
\end{lemma}

\begin{proof}
The statement is equivalent to:
\begin{equation*}
\mathrm{div}_{h_{\fre}} (\Theta)(e_l) = \dd \mathrm{Tr}(\Theta)(e_l)\, , \qquad  \mathrm{div}_{h_{\fre}} (\Theta)(e_n) = \dd \mathrm{Tr}_{h_{\fre}}(\Theta)(e_n)\, .
\end{equation*}

\noindent
Note that we indistinctly denote with the same symbol one-forms and their duals by the metric wherever no possible confusion may arise. Now we write:
\begin{equation*}
\Theta = \Theta_{ab}\, e_a\otimes e_b\, , \qquad \Theta_{ab}\in \mathrm{C}^\infty(\Sigma)\, , \qquad a, b = u,l,n\, ,
\end{equation*}

\noindent
where $\Theta_{ab}\in C^{\infty}(\Sigma)$ are smooth functions. Also, recall that by the definition of parallel Cauchy coframe, $\fre = (e_u,e_l,e_n)$ satisfies:
\begin{equation*}
\nabla^{h_{\fre}}_{e_b}e_a = - \delta_{au} \Theta(e_b) + \Theta(e_a,e_b) e_u\, .
\end{equation*}

\noindent
 Using the previous equation, we compute:
\begin{equation*}
\mathrm{div}_h (\Theta)(e_l) = \sum_a (\nabla^{h_{\fre}}_{e_a}\Theta)(e_a,e_l)  = \sum_a e_a (\Theta_{al}) - \sum_a \Theta_{ul} \Theta_{aa}\, ,
\end{equation*}
 
\noindent
as well as:
\begin{equation*}
\dd\text{Tr}(\Theta)(e_l) = \sum_a e_l(\Theta_{aa}) \, .
\end{equation*}

\noindent
Hence:
\begin{equation*}
\dd\text{Tr}_{h_{\fre}}(\Theta)(e_l) -\mathrm{div}_h (\Theta)(e_l) =-e_u(\Theta_{ul}) - e_n(\Theta_{ln}) + e_l(\Theta_{uu}) + e_l(\Theta_{nn}) +  \tul \mathrm{Tr}_{h_{\fre}}(\Theta)\, .
\end{equation*}

\noindent
Using now that $\dd^2 e_n=0$ we obtain $e_l(\tnn)-e_n(\tln)=\tln\tun-\tnn\tul$, which in turn implies $\dd\mathrm{Tr}(\Theta)(e_l) = \mathrm{div}_h (\Theta)(e_l)$. Similarly $\mathrm{div}_h (\Theta)(e_n) = \dd \mathrm{Tr}(\Theta)(e_n)$ and hence we conclude.
\end{proof}

\noindent
For further reference, we obtain the Ricci tensor and scalar curvature of the Riemannian metric $h_{\fre}$ associated to a parallel Cauchy pair $(\fre,\Theta)$.  

\begin{proposition}
Let  $(\fre,\Theta)$ be a parallel Cauchy pair. The Ricci curvature of $h_{\fre}$ is given by: 
\begin{equation}
\label{eq:ricci}
\mathrm{Ric}^{\fre}=\Theta \circ \Theta -\mathrm{Tr}_{\fre}(\Theta)\Theta + (\dd \mathrm{Tr}_{\fre}(\Theta) - \mathrm{div}_{\fre} (\Theta) ) \otimes e_u + \nabla^{\fre}_{e_u}\Theta - (\nabla^{\fre} \Theta)(e_u)\, ,
\end{equation}

\noindent
whereas the scalar curvature of $h_{\fre}$ reads:
\begin{equation}
\label{eq:scalarcurvature}
\mathrm{R}^{\fre}=\vert \Theta \vert^2_{\fre}-\mathrm{Tr}_{\fre}(\Theta)^2-2(\mathrm{div}_{\fre} (\Theta)(e^{\sharp}_u) -  \dd \mathrm{Tr}_{\fre}(\Theta)(e^{\sharp}_u) )\, ,
\end{equation}
where $\dd_{\nabla^h}$ denotes the exterior covariant derivative associated to $\nabla^h$.
\end{proposition}
 
\begin{proof}
The result is proven through a direct computation using the fact that for a parallel Cauchy pair $(\fre,\Theta)$ we have:
\begin{equation*}
\nabla^{h_{\fre}} e^{\sharp}_a = \Theta(e_a^{\sharp})\otimes e^{\sharp}_u - \delta_{ua} \Theta^{\sharp} \, , \qquad a = u,l,n\, 
\end{equation*}

\noindent
as well as:
\begin{eqnarray*}
& \nabla^{h_{\fre}}(\Theta(e_a^{\sharp})) = \sum_b(\dd \Theta_{ab}\otimes e_b  + \Theta_{ab} \Theta_b\otimes e_u) - \Theta_{u a } \Theta\, ,
\end{eqnarray*}

\noindent
where we have written $\Theta = \Theta_{ab} e_a\otimes e_b$, $a, b = u , l , n$. In our conventions the Ricci curvature reads:
\begin{equation*}
\mathrm{Ric}^{\fre} = \sum_c e^{\sharp}_u \lrcorner \dd_{\nabla^{\fre}}(e_c\otimes \Theta_c)- \sum_a e^{\sharp}_a \lrcorner \dd_{\nabla^{\fre}}(\Theta_a\otimes e_u)\, ,
\end{equation*}

\noindent
where $\dd_{\nabla^{\fre}}$ denotes the exterior covariant derivative for one-forms taking values on one forms. Expanding the desired result for the Ricci tensor follows. Taking the trace of Equation \eqref{eq:ricci}, we obtain:
\begin{equation*}
\mathrm{R}^h=\vert \Theta \vert^2-\mathrm{Tr}(\Theta)^2 - \mathrm{div}_h (\Theta)(e^{\sharp}_u) + \dd \mathrm{Tr}(\Theta)(e^{\sharp}_u)+\mathrm{Tr}_{\fre}(\nabla_{e_u} \Theta -(\nabla \Theta) (e_u))\, .
\end{equation*}

\noindent
The last term can be written as follows:
\begin{eqnarray*}
\sum_a ((\nabla_{e^{\sharp}_u} \Theta) (e^{\sharp}_a,e^{\sharp}_a) -(\nabla_{e^{\sharp}_a} \Theta) (e^{\sharp}_u,e^{\sharp}_a))  = \dd\mathrm{Tr}_{\fre}(\Theta)(e^{\sharp}_u) -   \mathrm{div}_{\fre} (\Theta)(e^{\sharp}_u)\, ,
\end{eqnarray*}

\noindent
whence:
\begin{equation*}
\mathrm{R}^{\fre}=\vert \Theta \vert^2-\text{Tr}(\Theta)^2-2(\mathrm{div}_h (\Theta)(e^{\sharp}_u) - \dd\mathrm{Tr}_{\fre}(\Theta)(e^{\sharp}_u) )\, ,
\end{equation*}

\noindent
and we conclude.
\end{proof}

\begin{remark}
If $\Theta$ is Codazzi then Equation \eqref{eq:ricci} simplifies to:
\begin{equation}
\label{eq:ricciCodazzi}
\mathrm{Ric}^{\fre}=\Theta \circ \Theta -\mathrm{Tr}_{\fre}(\Theta)\Theta\, ,
\end{equation}
 
\noindent 
which matches \cite[Proposition 5]{BaumMuller} modulo an unimportant constant factor.
\end{remark}

\begin{proposition}
\label{prop:hamiltonianmomentum}
A Cauchy pair $(\fre,\Theta)$ satisfies the Hamiltonian constraint, that is, the first equation in \eqref{eq:RicciflatonSigma}, if and only if $(\fre,\Theta)$ satisfies the momentum constraint, that is, the second equation in \eqref{eq:RicciflatonSigma}.  
\end{proposition}

\begin{proof}
Follows from the explicit expression \eqref{eq:scalarcurvature} for the scalar curvature of $h_{\fre}$ upon use of Lemma \ref{lemma:eumc}.
\end{proof}

\begin{proposition}
\label{prop:alleqs}
A pair $(\fre,\Theta)\in \Conf(\Sigma)$ is a constrained Ricci-flat parallel Cauchy pair if and only if:
\begin{eqnarray}
\label{eq:sys1}
&\dd e_u=\Theta(e_u) \wedge e_u \, , \quad \dd e_l=\Theta(e_l) \wedge e_u  \, , \quad \dd e_n=\Theta(e_n) \wedge e_u \, ,\\
& [\Theta(e_u)]=0\in H^1(\Sigma,\mathbb{R})\, , \quad  \mathrm{R}^{h_{\fre}}=\vert \Theta \vert^2-\mathrm{Tr}(\Theta)^2\,.
\label{eq:sys2}
\end{eqnarray}

\noindent
where $h_{\fre}$ is the Riemannian metric associated to $(\fre,\Theta)$. In particular, every Cauchy pair $(\fre,\Theta)$ whose shape operator $\Theta$ is Codazzi is constrained Ricci-flat.
\end{proposition}

\begin{proof}
By Proposition \ref{prop:hamiltonianmomentum} we only need to prove that if $(\fre,\Theta)$ is a parallel Cauchy pair and $\Theta$ is a Codazzi shape operator then the momentum constraint is automatically satisfied. Fix a point $p\in \Sigma$ and an orthonormal (with respect to $h_{\fre}$) frame $\left\{e_a\right\}$, $a = 1,2,3$, such that $\nabla^{h_{\fre}} e_a\vert_p = 0$. We compute at $p\in \Sigma$:
\begin{eqnarray*}
& \dd \mathrm{Tr}(\Theta) \vert_p = \sum_a \dd (\Theta(e_a,e_a))\vert_p = \sum_a (\nabla^{h_{\fre}}\Theta)(e_a,e_a)\vert_p + 2 \sum_a \Theta(\nabla^{h_{\fre}} e_a,e_a)\vert_p \\
& = \sum_a (\nabla^{h_{\fre}}_{e_a}\Theta)(e_a)\vert_p = \mathrm{div}_{h_{\fre}}(\Theta)\vert_p\, ,
\end{eqnarray*}

\noindent
and hence we conclude.	
\end{proof}

\begin{remark}
We will refer to a parallel Cauchy pair $(\fre, \Theta)$ whose shape operator is Codazzi as a \emph{Codazzi parallel Cauchy pair}. 
\end{remark}

\noindent
Proposition \ref{prop:alleqs} summarizes necessary conditions that a pair $(\fre,\Theta)$ needs to satisfy in order for the Lorentzian development of $(\Sigma,h_{\fre})$ to be a Ricci-flat Lorentzian four-manifold admitting a parallel spinor field. These conditions are satisfied by all examples in \cite{BaumMuller}.


\section{The topology and geometry of Cauchy pairs}
\label{sec:topologyCauchyPairs}


In this section we investigate the diffeomorphism and isometry type of oriented three-manifolds $\Sigma$ admitting a complete Cauchy pair $(\fre,\Theta)\in \Sol(\Sigma)$. 


\subsection{General considerations}

 
\begin{lemma}
\label{lemma:complete}
Let $(\fre,\Theta)$ be a complete Cauchy pair on $\Sigma$. The frame $\fre^{\sharp} = (e^{\sharp}_u , e^{\sharp}_l, e^{\sharp}_n)$ dual of $\fre$ is complete, that is, each of its elements is a complete vector field on $\Sigma$.  
\end{lemma}

\begin{proof}
Follows from the fact that $h_{\fre}$ is by assumption a complete Riemannian metric on $\Sigma$ respect to which each of the elements of $\fre$ has unit norm, see \cite[Page 154, Exercise 11]{Carmo} \footnote{Note however that there is a typo in Exercise 11, the correct condition being, using the notation of the exercise, $\vert X(p)\vert < c$ rather than $\vert X(p)\vert > c$.}. 
\end{proof}

\begin{lemma}
\label{lemma:flatfol}
Let $\fre = (e_u,e_l,e_n)$ be a complete Cauchy coframe. The distribution $\ker(e_u)\subset T\Sigma$ is integrable and defines a codimension one transversely orientable foliation in $(\Sigma,h_{\fre})$ whose leaves are complete and flat Riemann surfaces with respect to the metric induced by $h_{\fre}$.
\end{lemma}

\begin{proof}
The first equation in the Cauchy differential system \eqref{eq:exder} immediately implies:
\begin{equation*}
e_u \wedge \dd e_u=0\, ,
\end{equation*}
	
\noindent
and thus Cartan's criterion implies in turn that $\ker(e_u)\subset T\Sigma$ defines an integrable transversely orientable codimension one distribution, whose associated foliation we denote by $\cF_{\fre}$. Let $p \in \Sigma$ and denote by $\cF_{\fre , p} \subset \Sigma$ the maximal leaf of $\cF_{\fre}$ passing through $p$. The cotangent space of $\cF_{\fre , p}$ is spanned over $C^{\infty}(\cF_{\fre , p})$ by the restriction of $e_l$ and $e_n$:
\begin{equation*}
T^{\ast}\cF_{\fre , p} = \mathrm{Span}_{C^{\infty}(\cF_{\fre , p})}(e_l\vert_{T\cF_{\fre , p}},e_n\vert_{T\cF_{\fre , p}})\, .
\end{equation*}
	
\noindent
Furthermore:
\begin{equation*}
h_{\fre}\vert_{\cF_{\fre , p}} = e_l\vert_{T\cF_{\fre , p}}\otimes e_l\vert_{T\cF_{\fre , p}} + e_n\vert_{T\cF_{\fre , p}}\otimes e_n\vert_{T\cF_{\fre , p}}\, .
\end{equation*}

\noindent
A direct computation, using the fact that $\fre$ is a parallel Cauchy coframe, shows that $(e_l\vert_{T\cF_{\fre , p}},e_n\vert_{T\cF_{\fre , p}})$ is a flat coframe with respect to the Levi-Civita connection of the metric induced by $h_{\fre}$ whence $h_{\fre}\vert_{\cF_{\fre , p}}$ is flat. The fact that the leaves of $\cF_{\fre}$ equipped with the metric induced by $h_{\fre}$ are complete manifolds follows from completeness of $h_{\fre}$ and is proved explicitly in \cite[Proposition 1.26]{Schliebner}.
\end{proof}

\noindent
Since the leaves of $\cF_{\fre}$ are complete and flat they must be isometric to either the euclidean plane, the euclidean cylinder or a flat torus. As we will see momentarily, this poses strong constraints on the differentiable topology of $\Sigma$. Given a Cauchy pair $(\fre,\Theta)$, the cohomological condition occurring in the Cauchy differential system \eqref{eq:exder} guarantees that there exists a function $\mathfrak{f}\in C^{\infty}(\Sigma)$ such that:
\begin{equation*}
\Theta(e_u) = -\dd \mathfrak{f}\, .
\end{equation*}

\noindent
Therefore, by the first equation in \eqref{eq:exder}, the one-form $\hat{e}_u := e^{\mathfrak{f}}e_u\in \Omega^1(\Sigma)$ is closed and satisfies $\Ker(\hat{e}_u) = \Ker(e_u)$, implying that we can consider $\cF_{\fre}\subset \Sigma$ as a foliation defined by the kernel of the nowhere vanishing closed one-form $\hat{e}_u$, a type of foliation that has been extensively studied in the literature, see for example \cite{Conlon,Tondeur}. It can be easily seen that the metric $h_{\fre}$ will not be, in general, bundle-like with respect to $\cF_{\fre}$. On the other hand, given a Cauchy pair $(\fre,\Theta)$, the following \emph{modified} Riemannian metric:
\begin{equation*}
h_{\hat{\fre}} = \hat{e}_u\otimes \hat{e}_u + e_l\otimes e_l + e_n\otimes e_n\, , \qquad \hat{\fre} = (\hat{e}_u ,e_l ,e_n)\, ,
\end{equation*}

\noindent
is indeed bundle-like, that is, it satisfies the following condition:
\begin{equation*}
\cL_{v} h_{\hat{\fre}}\vert_{ T\cF_{\fre}^{\perp_{h_{\fre}}}} = 0\, , \qquad \forall\,\, v\in \Gamma(T\cF_{\fre})\, .
\end{equation*}

\noindent
In other words, $h_{\hat{\fre}}\vert_{ T\cF_{\fre}^{\perp_{h_{\fre}}}}$ is a holonomy invariant transversal metric.

\begin{remark}
By Lemma \ref{lemma:complete}, $e^{\sharp}_u \in \mathfrak{X}(\Sigma)$ is a complete vector field on $\Sigma$. However, the same statement may not hold for $\hat{e}_u^{\sharp}\in \mathfrak{X}(\Sigma)$, the metric dual of $\hat{e}_u$ with respect to $\hat{h}_{\fre}$.
\end{remark}

\begin{definition}
A Cauchy pair $(\fre,\Theta)$ is \emph{fully complete} if it is complete and in addition $\hat{e}_u^{\sharp}\in \mathfrak{X}(\Sigma)$ is complete. 
\end{definition}

\noindent
The notion of fully complete Cauchy pair is convenient to obtain global results about Cauchy pairs by using completeness of $\hat{e}_u$ to identify the leaves of $\cF_{\fre}\subset \Sigma$.  

\begin{proposition}
\label{prop:generalformuniv}
Let $(\fre,\Theta)$ be a fully complete Cauchy pair on $\Sigma$ with associated foliation $\cF_{\fre}\subset \Sigma$. The following holds:

\begin{enumerate}[leftmargin=*]
\item All leaves are diffeomorphic to a model leaf given by either the plane $\mathbb{R}^2$, the cylinder or the torus.
		
\item Either all leaves are closed or all leaves are dense in $\Sigma$.
		
\item The Riemannian universal cover of $(\Sigma,h_{\fre})$ is isometric to $(\mathbb{R}^3, \bar{h}_{\fre})$ with metric $\bar{h}_{\fre}$ given by:
\begin{equation*}
\bar{h}_{\fre} \eqdef e^{2\mathfrak{u}} \dd x\otimes \dd x + \mathfrak{h}_x\, ,
\end{equation*} 
		
\noindent
where $x$ is the first Cartesian coordinate of $\mathbb{R}^3$, $\mathfrak{u}\in C^{\infty}(\mathbb{R}^3)$ is a smooth function and, for every $x\in \mathbb{R}$, $h_x$ is a flat euclidean metric on $\left\{ x\right\}\times \mathbb{R}^2 \subset \mathbb{R}^3$. If $(\fre,\Theta)$ is not fully complete the previous characterization is only guaranteed to hold locally.
\end{enumerate}
\end{proposition}

\begin{remark}
Item (3) in the previous proposition recovers, in the specific case of four Lorentzian dimensions, items (1) and (2) in \cite[Theorem 4]{LeistnerLischewski}. Indeed, the function $e^{2\mathfrak{u}}$ can be shown to determine the norm of the one-form $u^{\perp}$ occurring in the original formulation of the constraint equations \eqref{eq:cauchyul}. The norm of $u^{\perp}$ is denoted by $u^2$ in Op. Cit. 
\end{remark}

\begin{proof}
\emph{Bar} over a symbol will denote lift to the universal cover of $\Sigma$, denoted by $\bar{\Sigma}$. We prove the proposition point by point:

\begin{enumerate}[leftmargin=*]
	\item Since $\cF_{\fre}$ is defined by a closed nowhere-vanishing one-form the fact that all its leaves must be diffeomorphic is classical, see \cite{MilnorII,Conlon}. To motivate it, recall that the Lie derivative of $\hat{e}_u$ along $\hat{e}^{\sharp}_u$ is zero (note that this is not true in general for $e_u$ and $e^{\sharp}_u$). Hence, the flow $(\psi_t)_{t\in \mathbb{R}}$ defined by the complete vector field $\hat{e}^{\sharp}_u$ preserves the leaves of $\cF_{\fre}$, that is, maps leaves to leaves diffeomorphically. Furthermore, for every $p,q\in \Sigma$ there exists a $t_0\in \mathbb{R}$ such that:
	\begin{equation*}
	\psi_{t_0}\vert_{\cF_{\fre,p}}\colon \cF_{\fre,p} \to \cF_{\fre,q}\, .
	\end{equation*}
	
	\noindent
	Hence, all leaves of $\cF_{\fre}$ are diffeomorphic and by Lemma \ref{lemma:flatfol} they must be all diffeomorphic to either the plane, the cylinder or the torus.
	
	\item Follows from \cite[Proposition 5.1]{ConlonII}.
	
	\item  The fact that the universal cover $\bar{\Sigma}$ is diffeomorphic to $\mathbb{R}\times \bar{\cF}_{\fre}$, where $\bar{\cF}_{\fre}$ denotes the universal cover of the typical leaf of $\cF_{\fre}$ is proven in detail in \cite[Proposition 8]{LeistnerSchliebner}. Furthermore, the foliation $\cF_{\fre}\subset \Sigma$ lifts to the foliation whose leaves are given by $\left\{ x\right\}\times \bar{\cF}_{\fre}\subset \mathbb{R}\times \bar{\cF}_{\fre}$ for $x\in \mathbb{R}$. Since the typical leaf of $\cF_{\fre}$ is either the plane, the cylinder or the torus, then $\bar{\cF}_{\fre} = \mathbb{R}^2$ and therefore $\bar{\Sigma} = \mathbb{R}^3$. The lift $\bar{e}_u$ of $e_u$ to $\bar{\Sigma}$ is orthogonal to $T^{\ast}\bar{\cF}_{\fre}\subset T^{\ast}\bar{\Sigma}$, whence:
	\begin{equation*}
	\bar{e}_u = e^{\fru} \dd x\, , \qquad \fru \in C^{\infty}(\mathbb{R}^3)\, ,
	\end{equation*}
	
	\noindent
	where $\mathfrak{u}$ is a function on $\mathbb{R}^3$ satisfying $\bar{\Theta}(\bar{e}_u) = - \dd\fru$. Since the distribution $T\bar{\cF}_{\fre}\subset T\bar{\Sigma}$ is defined by the kernel of $\bar{e}_u$ we conclude that the lift of $h_{\fre}$ to $\bar{\Sigma}$ can be written as follows:
	\begin{equation*}
	\bar{h}_{\fre} \eqdef e^{2\mathfrak{u}} \dd x\otimes \dd x + \mathfrak{h}_x\, ,
	\end{equation*}
	
	\noindent
	for a family $\left\{ \mathfrak{h}_x \right\}_{x\in \mathbb{R}}$ of two-dimensional metrics on $\mathbb{R}^2$, which must be flat by Lemma \ref{lemma:flatfol}.
\end{enumerate}
\end{proof}

\noindent
The leaves of the foliation $\cF_{\fre}$ are all mutually diffeomorphic but a priori may not be mutually isometric since (the dual of) $\hat{e}_u$ which generates the flow that allows to identify different leaves of $\cF_{\fre}$ may not be an isometry of $h_{\fre}$. We will refer to the type of any leaf of $\cF_{\fre}$ as the \emph{typical leaf} of $\cF_{\fre}$, considered as a Riemann surface with the induced orientation. If the typical leaf of $\cF_{\fre}$ is compact we obtain the following result.

\begin{proposition}
Let $(\fre,\Theta)$ be a fully complete Cauchy pair on $\Sigma$ with associated foliation $\cF_{\fre}\subset \Sigma$. If the typical leaf of $\cF_{\fre}$ is a flat torus then either $\Sigma = \mathbb{R}\times T^2$ or $\Sigma$ admits the structure of a fiber bundle $\pi_{\fre} \colon \Sigma \to S^1$ inducing $\cF_{\fre}$.
\end{proposition}

\begin{proof}
Follows directly from \cite[Corollary 8.6]{Sharpe} by using the fact that every locally trivial fibration over $\mathbb{R}$ is trivial as well as the fact that if the leaves of $\cF_{\fre}$ are compact then they must be diffeomorphic to the torus.
\end{proof}

\begin{lemma}
\label{lemma:actionR2}
Let $(\fre,\Theta)$ be a complete Cauchy pair on $\Sigma$ with associated foliation $\cF_{\fre}\subset \Sigma$. Then, $\Sigma$ admits a canonical locally free action of $\mathbb{R}^2$ whose orbits are the leaves of $\cF_{\fre}$. 
\end{lemma}

\begin{proof}
Consider a Cauchy pair $(\fre,\Theta)$ and define the map:
\begin{equation*}
\Psi \colon \mathbb{R}^2 \times \Sigma \to \Sigma\, , \qquad (t_1,t_2,p) \mapsto \Phi^{t_1}_{e_l}\circ \Phi^{t_2}_{e_n}(p)\, ,
\end{equation*}

\noindent
where $\Phi^{t_1}_{e_l}$ (respectively $\Phi^{t_2}_{e_n}$) denotes the flow generated by $e^{\sharp}_l$ (respectively $e^{\sharp}_n$) at the \emph{time} $t_1$ (respectively $t_2$). Using that $\fre$ is a solution of the Cauchy differential system, we obtain:
\begin{equation*}
[e_l^{\sharp} , e_n^{\sharp}] = \nabla^{h_{\fre}}_{e_l^{\sharp}} e_n^{\sharp} - \nabla^{h_{\fre}}_{e_n^{\sharp}} e_l^{\sharp} = 0\, ,
\end{equation*}

\noindent
hence $\Psi$ defines a smooth action of $\mathbb{R}^2$ on $\Sigma$, which, since both $e_l$ and $e_n$ are nowhere vanishing, is locally free. Furthermore, the fact that $e_l^{\sharp}$ and $e_n^{\sharp}$ are complete and span $T\cF_{\fre} \subset T\Sigma$ implies that the orbits of $\Phi$ correspond to the leaves of $\cF_{\fre}$.
\end{proof}

\noindent
Locally free actions of the group $\mathbb{R}^2$ on three-manifolds have been extensively studied extensively in the literature, see \cite{ArrautCraizer,ChateletRW,Rosenberg,RosenbergRW} and references therein, especially in relation with the problem of finding the number of nowhere vanishing and everywhere linearly independent commuting vector fields on a compact three-manifold.

\begin{proposition}
Let $(\fre,\Theta)$ be a fully complete Cauchy pair on $\Sigma$ such that the restriction of $\Theta$ to $T\cF_{\fre}\subset T\Sigma$ vanishes, that is, $\Theta\vert_{T\cF_{\fre}\times T\cF_{\fre}} = 0$. Then, $\Sigma$ is diffeomorphic to $T^{k}\times \mathbb{R}^{3-k}$ for some integer $k\in \left\{ 0,1,2,3\right\}$.
\end{proposition}

\begin{proof}
Let $(\fre,\Theta)$ be a Cauchy pair such that  $\Theta\vert_{T\cF_{\fre}\times T\cF_{\fre}} = 0$. Then, $\hat{\fre}^{\sharp}$ is a global frame of commuting vector fields, which can be used to define a smooth action of $\mathbb{R}^3$ on $\Sigma$ exactly as it occurred in the proof of Lemma \ref{lemma:actionR2} to define an action of $\mathbb{R}^2$. Since $\hat{\fre}$ is assumed to be complete, this action is transitive. The final step of the proof consist in showing that the stabilizer of the action is of the form $\mathbb{Z}^k\times \left\{ 0\right\} \subset \mathbb{R}^k\times \mathbb{R}^{3-k}$ acting naturally on $\mathbb{R}^3$. This is explicitly proven in \cite[Chapter 4]{Conlon}.
\end{proof}


\subsection{Complete Cauchy pairs on the universal Riemannian cover}


Let $(\fre,\Theta)$ be a fully complete Cauchy pair on $\Sigma$. Proposition \ref{prop:generalformuniv} states that universal Riemannian cover of $(\Sigma,h_{\fre})$ is isometric to $\mathbb{R}^3$ when the latter is equipped with the metric:
\begin{equation}
\label{eq:metricR3}
\bar{h}_{\fre} \eqdef e^{2\mathfrak{u}} \dd x\otimes \dd x + \frh_x\, ,
\end{equation} 

\noindent
where $\frh_x$ is a flat metric on $\left\{ x\right\}\times \mathbb{R}^2 \subset \mathbb{R}^3$ for every $x\in \mathbb{R}$. The corresponding Cauchy coframe reads:
\begin{equation}
\label{eq:CauchyframeR3}
e_u = e^{\fru} \dd x \, , \qquad e_l = e_l(x) \, , \qquad e_n = e_n(x)\, ,
\end{equation}

\noindent
where $e_l$ and $e_n$ depend only on the coordinate $x$. A quick computation shows that the exterior derivative of this frame is given by:
\begin{equation*}
\dd e_u = \dd\fru \wedge e_u \, , \qquad \dd e_l =   e^{-\fru} e_u\wedge \cL_{x} e_l \, , \qquad \dd e_n =   e^{-\fru} e_u\wedge \cL_{x} e_n\, ,
\end{equation*}

\noindent
where the symbol $\cL_{x}$ denotes Lie derivative with respect to $\partial_x$. Plugging the previous equations into the Cauchy differential system \eqref{eq:exder} we obtain the following lemma.

\begin{lemma}
\label{lemma:coframeR3}
A pair $(\fre,\Theta)\in \Conf(\mathbb{R}^3)$, where $\fre$ is given by the coframe \eqref{eq:CauchyframeR3}, is a Cauchy pair if and only if the following equations are satisfied:
\begin{equation*}
(\dd \fru - \Theta(e_u))\wedge e_u = 0\, , \quad (\Theta(e_l) +  e^{-\fru} \cL_{x} e_l )\wedge e_u = 0\, , \quad (\Theta(e_n) +  e^{-\fru} \cL_{x} e_n )\wedge e_u = 0\, .
\end{equation*}
\end{lemma}

\noindent
The previous lemma is used in the following theorem to \emph{solve} the shape operator of a parallel Cauchy pair $(\fre,\Theta)$ defined on a connected and simply connected three-manifold $\Sigma$ in terms of the Cauchy coframe $\fre$.  

\begin{theorem}
\label{thm:universalcover}
A pair $(\fre,\Theta)\in \Conf(\Sigma)$ is a parallel and fully complete Cauchy pair on a connected and simply connected three-manifold $\Sigma$ if and only if there exist global coordinates $(x,y,z)$ identifying $\Sigma = \mathbb{R}^3$ such that $\fre$ satisfies:
\begin{equation}
\label{eq:mixedcondition}
\fre = (e^{\fru} \dd x , e_l(x) , e_n(x))\, , \quad (\cL_{x} e_l)(e^{\sharp}_n) =  (\cL_{x} e_n)(e^{\sharp}_l)\, ,
\end{equation}

\noindent
and in addition:
\begin{equation}
\label{eq:ThetaR3}
\Theta = (\mathfrak{F}(x)\, e^{-\fru} + \partial_x e^{-\fru} )\, e_u\otimes e_u + e_u\otimes \dd \fru + \dd\fru \otimes e_u - \frac{1}{2} e^{-\fru} \cL_x h_x\, ,
\end{equation}

\noindent
where $\mathfrak{F}\in C^{\infty}(\mathbb{R})$ is a function of $x$.  
\end{theorem}

\begin{remark}
The second equation in \eqref{eq:mixedcondition} is non-trivial in general and hence restricts the type of coframes that can occur as part of a parallel Cauchy pair. 
\end{remark}

\begin{proof}
Let $(\fre,\Theta)$ be a Cauchy pair on a connected and simply connected three-manifold $\Sigma$. The fact that there exist global coordinates $(x,y,z)$ identifying $\Sigma$ with $\mathbb{R}^3$ respect to which $\fre$ is given by:
\begin{equation*}
\fre = (e^{\fru} \dd x , e_l(x) , e_n(x))\, ,
\end{equation*}

\noindent
follows directly from Proposition \ref{prop:generalformuniv}. On the other hand, Lemma \ref{lemma:coframeR3} implies:
\begin{equation*}
\Theta(e_u) =\dd\fru + f_u\, e_u\, , \quad \Theta(e_l) = f_l\, e_u - e^{-\fru} \cL_{x} e_l  \, , \quad \Theta(e_n) = f_n\, e_u - e^{-\fru} \cL_{x} e_n\, ,
\end{equation*} 

\noindent
for functions $f_u , f_l , f_n \in C^{\infty}(\mathbb{R}^3)$. Symmetry of $\Theta$ is equivalent to the following equations:
\begin{equation*}
f_l = \dd\fru(e^{\sharp}_l)\, , \qquad f_n = \dd\fru(e^{\sharp}_n)\, , \qquad   (\cL_{x} e_l)(e^{\sharp}_n) =  (\cL_{x} e_n)(e^{\sharp}_l
)\, .
\end{equation*}

\noindent
These conditions imply that $\Theta$ must be of the form:
\begin{equation}
\label{eq:ThetaR3proof}
\Theta = (f_u-\dd\fru(e^{\sharp}_u))\, e_u\otimes e_u + e_u\otimes \dd \fru + \dd\fru \otimes e_u - \frac{1}{2} e^{-\fru} \cL_x h_x\, ,
\end{equation}

\noindent
Furthermore, the fact that $\Theta(e_u)$ must be closed, whence exact, is equivalent to:
\begin{equation*}
\dd (f_u\,e_u) = \dd(f_u e^{\fru})\wedge \dd x = 0\, .
\end{equation*}

\noindent
Therefore, $f_u e^{\fru} = \mathfrak{F}(x)$ for a smooth function $\mathfrak{F}$ depending exclusively on the coordinate $x$. Plugging this expression back in \eqref{eq:ThetaR3proof} we obtain \eqref{eq:ThetaR3}. The converse follows by construction and can be verified explicitly by inserting \eqref{eq:ThetaR3} in the parallel Cauchy differential system \eqref{eq:exder}.
\end{proof}

\begin{remark}
\label{remark:universalcover}
Theorem \ref{thm:universalcover} recovers \cite[Theorem 4]{LeistnerLischewski} in the language of parallel Cauchy pairs and in the specific case of four Lorentzian dimensions, refining it and providing an alternative proof of the result. The refinement is contained in the \emph{extra} information provided by the Cauchy coframe $\fre$, which needs to satisfy equations \eqref{eq:mixedcondition}. On the other hand, equation \eqref{eq:ThetaR3} does not specify uniquely $\Theta$ but allows the freedom of choosing the arbitrary function $\cF(x)$. This arbitrary function seems to be absent in \cite[Theorem 4]{LeistnerLischewski}.
\end{remark}

\begin{example}
Using the notation and framework established by Theorem \ref{thm:universalcover}, assume that:
\begin{equation*}
\mathfrak{h}_x =  e^{2\frw(x)}( \dd y\otimes \dd y + \dd z\otimes \dd z)\, , 
\end{equation*}

\noindent
where $(x,y,z)$ are the Cartesian coordinates of $\mathbb{R}^3$ and $\frw(x)$ is a function on $\mathbb{R}^3$ depending only on the coordinate $x$. As defined above, $\mathfrak{h}_x$ is clearly a family of flat metrics on $\mathbb{R}^2$ parametrized by $x\in \mathbb{R}$. The corresponding parallel Cauchy coframe reads:
\begin{equation}
\fre = (e^{\fru} \dd x , e^{\frw(x)} \dd y  , e^{\frw(x)} \dd z )\, ,  
\end{equation}

\noindent
One easily checks that the second equation in \eqref{eq:mixedcondition} is automatically satisfied. On the other hand, the corresponding parallel shape operator is given by:
\begin{equation*}
\Theta = (\mathfrak{F} \, e^{-\fru} + \partial_x e^{-\fru} )\, e_u\otimes e_u + e_u\otimes \dd \fru + \dd\fru \otimes e_u - \partial_x\frw(x)\, e^{-\fru} 	\mathfrak{h}_x\, .
\end{equation*}

\noindent
Using the previous expression, we compute:
\begin{eqnarray*}
& \mathrm{Tr}_{\fre}(\Theta) = e^{-\fru} (\mathfrak{F}  + \partial_x \fru - 2 \partial_x \frw )\, , \\ & \vert\Theta\vert^2_{\fre} = e^{-2\fru} ((\mathfrak{F}  + \partial_x \fru)^2 + 2 (\partial_x \frw )^2) + 2 e^{-2\frw } ((\partial_y \fru)^2 + (\partial_z \fru)^2)\, .
\end{eqnarray*}

\noindent
In particular:
\begin{equation*}
\vert\Theta\vert^2_{\fre} - \mathrm{Tr}_{\fre}(\Theta)^2 =  2 e^{-2\fru} \partial_x \frw(x) (2 (\mathfrak{F}(x) + \partial_x \fru) - \partial_x \frw(x))+ 2 e^{-2\frw } ((\partial_y \fru)^2 + (\partial_z \fru)^2)\, ,
\end{equation*}

\noindent
and since the scalar curvature of $h_{\fre}$ is given by:
\begin{equation*}
\mathrm{R}^{\fre} = e^{-2 \fru} \left(4 \partial_x \frw \partial_x \fru-4  \partial_x^2 \frw-6 (\partial_x\frw)^2\right)-2 e^{-2 \frw} \left((\partial_y \fru)^2+(\partial_z \fru)^2+\partial_y^2 \fru+\partial_z^2 \fru\right)\, ,
\end{equation*}

\noindent
we conclude that such parallel Cauchy pair $(\fre,\Theta)$ is constrained Ricci-flat if and only if:
\begin{equation*}
2 e^{2 \frw} \left(\mathfrak{F} \partial_x \frw  +   \partial_x^2 \frw + (\partial_x\frw)^2\right) +   e^{2 \fru} \left(2(\partial_y \fru)^2+ 2(\partial_z \fru)^2+\partial_y^2 \fru+\partial_z^2 \fru\right) = 0 \, .	 
\end{equation*}

\noindent
If the second term in the previous equation only depends on $x$ and $\partial_x \frw \neq 0 $ everywhere, then we can always solve it by choosing $\mathfrak{F}$ as follows:
\begin{equation*}
\mathfrak{F} = - \frac{1}{\partial_x \frw} \left( \partial_x^2 \frw + (\partial_x\frw)^2\right) -   \frac{e^{2 (\fru - \frw)}}{2 \partial_x\frw} \left(2(\partial_y \fru)^2+ 2(\partial_z \fru)^2+\partial_y^2 \fru+\partial_z^2 \fru\right) \, .
\end{equation*}
\end{example}


\subsection{Parallel Cauchy pairs on compact three-manifolds}


In this section we consider the isometry type of Cauchy pairs on closed three-manifolds, commenting briefly on the compact case with boundary. 

\begin{proposition}
Let $\Sigma$ be an oriented closed three-manifold admitting a Cauchy pair $(\fre,\Theta)$. Then $\Sigma$ is diffeomorphic to a torus bundle over $S^1$, that is, it is diffeomorphic to the suspension $\cX_{\mathfrak{k}}$ of $T^2$ by an element $\mathfrak{k}\in\mathrm{SL}(2,\mathbb{Z})$.
\end{proposition}

\begin{proof}
Let $(\fre,\Theta)$ be a Cauchy pair on $\Sigma$. By Lemma \ref{lemma:actionR2} $\Sigma$ admits locally free action of $\mathbb{R}^2$. Reference \cite{RosenbergRW} proves that $\Sigma$ admits such an action if and only if $\Sigma$ is diffeomorphic to a locally trivial torus bundle over $S^1$, which can always be constructed as a suspension of $T^2$ by an element $\mathfrak{k}\in\mathrm{SL}(2,\mathbb{Z})$ acting linearly on $T^2$.
\end{proof}

\noindent
Since it will be of importance in the following, we briefly recall the suspension construction of a torus bundle over $S^1$, which depends on a choice of orientation preserving diffeomorphism of $T^2$ modulo homotopy equivalence. Since $\Diff(T^2)$ is homotopy equivalent to $\Sl(2,\mathbb{Z})$ acting \emph{linearly} on $T^2$, it is enough to consider elements in $\Sl(2,\mathbb{Z})$. Let $\mathfrak{k}\in \mathrm{SL}(2,\mathbb{Z})$ and denote by $\langle \mathfrak{k}\rangle\subset \Sl(2,\mathbb{Z})$ the cyclic group generated by the element $\mathfrak{k}$. There exists a natural properly discontinuous fixed point free action of $\langle \mathfrak{k}\rangle$ on $\mathbb{R}\times T^2$ given by:
\begin{equation}
\label{eq:Sl2Zaction}
\mathfrak{k}\cdot (z,v) = (z+1, \mathfrak{k}(v))\, , \qquad (z,v)\in \mathbb{R}\times T^2\, ,
\end{equation}

\noindent
where $\mathfrak{k}$ acts linearly on $\mathbb{R}^2/\mathbb{Z}^2$. The suspension of $\mathbb{R}\times T^2$ by $\mathfrak{k} \in \Sl(2,\mathbb{Z})$ is by definition the quotient:
\begin{equation*}
\cX_{\mathfrak{k}} = \frac{\mathbb{R}\times T^2}{\langle \mathfrak{k}\rangle}\, , 
\end{equation*}

\noindent
equipped with the projection:
\begin{equation*}
\pi \colon \cX_{\mathfrak{k}} \to S^1 = \mathbb{R}/\mathbb{Z}\, , \qquad [z,v]\mapsto [z]\, .
\end{equation*}

\noindent
Equivalently, $\cX_{\mathfrak{k}}$ can be constructed by gluing $\left\{ 0 \right\}\times T^2$ and $\left\{ 1 \right\}\times T^2$ in $[0,1]\times T^2$ through the diffeomorphism $\mathfrak{k}\colon T^2\to T^2$. The element $\mathfrak{k}\in \Sl(2,\mathbb{R})$ determines completely the topology of $\cX_{\mathfrak{k}}$ and in particular determines if a given foliation of $\cX_{\mathfrak{k}}$ admits a bundle-like metric. Note that, given a Cauchy pair $(\fre,\Theta)$ on $\Sigma = \cX_{\mathfrak{k}}$,  the leaves of the foliation $\cF_{\fre}\subset \cX_{\mathfrak{k}}$ will not coincide in general with the fibers of $\cX_{\mathfrak{k}}$. We summarize now two important methods for constructing foliations in $\cX_{\mathfrak{k}}$.  

\begin{itemize}[leftmargin=*]
\item \emph{Linear plane foliations on $T^3$}. Denote by $\Diff(S^1)$ the group of orientation preserving diffeomorphisms of $S^1$ and consider the three-manifold  $\mathbb{R}^2\times S^1$. Fix a representation:
\begin{equation*}
\rho = (\rho_a,\rho_b) \colon \pi_1(T^2) = \mathbb{Z} \oplus \mathbb{Z} \to \Diff(S^1)\, ,
\end{equation*}

\noindent
such that the rotational numbers $r_a\in S^1$ and $r_b\in S^1$ of $\rho_a(1)$ and $\rho_b(1)$ are both irrational and rationally independent. Then, $\rho_a(1)$ and $\rho_b(1)$ generate a subgroup of the orientation preserving diffeomorphism group $\Diff(S^1)$, which we denote by:
\begin{equation*}
\langle \rho_a(1) , \rho_b(1)\rangle \subset \Diff(S^1)\, .
\end{equation*}

\noindent 
There is a canonical fixed point free action of $\langle \rho_a(1) , \rho_b(1)\rangle$ on $\mathbb{R}^2\times S^1$ given by:
\begin{equation*}
\rho_a(1)\cdot (x_1,x_2,\theta) = (x_1 + 1,x_2, \rho_a(1)(\theta))\, , \qquad \rho_b(1)\cdot (x_1,x_2,\theta) = (x_1,x_2+1, \rho_b(1)(\theta))\, ,
\end{equation*}

\noindent
on the generators $\rho_a(1)$ and $\rho_b(1)$. The quotient:
\begin{equation*}
\cX_{\rho} := \mathbb{R}^2\times S^1/\langle \rho_a(1) , \rho_b(1)\rangle\, ,
\end{equation*}

\noindent
of $\mathbb{R}^2\times S^1$ by the previous action is diffeomorphic to $T^3$ and the plane foliation of $\mathbb{R}^2\times S^1$ whose leaves are embedded planes $\mathbb{R}^2\times \left\{\theta\right\}\subset \mathbb{R}^2\times S^1$, $\theta\in S^1$, descends to a foliation by planes of  $\mathbb{R}^2\times S^1/\langle \rho_a(1) , \rho_b(1)\rangle$, which is called the \emph{suspension foliation} defined by $\rho$ and it is denoted by:
\begin{equation*}
\cF_\rho\subset \cX_{\rho}  = \mathbb{R}^2\times S^1/\langle \rho_a(1) , \rho_b(1)\rangle\, .
\end{equation*}

\noindent
In particular, $\cX_{\rho}$ admits the structure of a $S^1$ bundle over $T^2$ transverse to $\cF_\rho$, which is obtained by the standard associated bundle construction. Note that $\rho_a(1)$ and $\rho_b(1)$ may not be rotations of $S^1$ by a constant angle. In general, the foliation $\cF_\rho$ is only $C^0$ isomorphic to a foliation for which $\rho_a(1)$ and $\rho_b(1)$ are rotations, see \cite{Herman} for a explicit counterexample. However, if $\cF_\rho$ is defined by a non-singular closed one-form then $\cF_\rho$ is at least $C^1$ isomorphic to a foliation for which $\rho_a(1)$ and $\rho_b(1)$ are rotations \cite{Herman}. 

\item \emph{Cylinder foliations of circle bundles.} Consider the foliation $\cF_0\subset T^2\times\mathbb{R}$ whose leaves are defined to be the embedded submanifolds $\left\{\theta_1\right\} \times S^1\times \mathbb{R} \subset T^2\times \mathbb{R} = S^1\times S^1\times\mathbb{R}$ for $\theta_1 \in S^1$. For every diffeomorphism $\mathfrak{f}\colon T^2\to T^2$ preserving the foliation by standard circles $\left\{\theta_1\right\}\times S^1 \subset T^2 = S^1\times S^1 $ and such that its restriction to the first circle factor $\frf \vert_{S^1 \times \{ \theta_2\}}: S^1 \to S^1$ has an irrational rotation number, we define a diffeomorphism of $T^2\times \mathbb{R}$ as follows:
\begin{equation}
\label{eq:frfdefinitioncyl}
T^2\times \mathbb{R} \to T^2\times \mathbb{R}\, , \qquad (\theta_1 , \theta_2, x) \mapsto (\mathfrak{f}(\theta_1 , \theta_2), x+1)\, .
\end{equation}

\noindent
By \cite[Theorem 2]{ChateletRW} and \cite[Page 254 Th\'eor\`eme 1]{Hector}  $\mathfrak{f}\in \Sl(2,\mathbb{Z})$ is conjugate to an element of the form:
\begin{equation}
\label{eq:frfsl2z}
\begin{pmatrix}
	1 & n \\
	0 & 1 
\end{pmatrix} \in \Sl(2,\mathbb{Z})\, ,
\end{equation}

\noindent
where $n\in \mathbb{Z}$ is an integer. Denote by $\langle \mathfrak{f}\rangle \subset \Diff(T^2\times \mathbb{R})$ the cyclic subgroup of $\Diff(T^2\times \mathbb{R})$ generated by the previous action, and define:
\begin{equation*}
\cX_{\mathfrak{f}} := \frac{\mathbb{R}\times T^2}{\langle\mathfrak{f}\rangle}\, ,
\end{equation*}

\noindent
to be the quotient of $\mathbb{R}\times T^2$ by $\langle \mathfrak{f}\rangle $, which defines a fiber bundle $\pi_{\mathfrak{f}}\colon \cX_{\mathfrak{f}}\to S^1$ with projection:
\begin{equation*}
\pi_{\mathfrak{f}}([\theta_1 , \theta_2, x]) = [x] \in S^1\, .
\end{equation*}

\noindent
We see that the action of $\langle \mathfrak{f}\rangle$ preserves by construction $\cF_0$, whence $\cF_0$ descends to a foliation $\cF_{\mathfrak{f}} \subset \cX_{\mathfrak{f}}$ whose fibers are all diffeomorphic to the cylinder. More explicitly, the leaves of the foliation are given by:
\begin{equation*}
\mathrm{p}_{\mathfrak{f}}(\left\{\theta \right\}\times S^1\times \mathbb{R})\subset \cX_{\mathfrak{f}}\, , \quad \theta \in S^1\, ,
\end{equation*}
 
\noindent
where $\mathrm{p}_{\mathfrak{f}} \colon T^2\times \mathbb{R} \to \cX_{\mathfrak{f}}$ denotes the canonical projection.
\end{itemize}

\begin{proposition}
Every codimension-one foliation of $\cX_{\mathfrak{k}}$ defined by the kernel of a nowhere vanishing closed one-form whose leaves are all diffeomorphic to either the plane $\mathbb{R}^2$ or the cylinder $\mathbb{R}^2\backslash\left\{ 0\right\}$ is isomorphic to one of the foliations defined above. 
\end{proposition}

\begin{remark}
By \emph{isomorphic foliations} we mean foliations for which there exists a $C^1$ diffeomorphism between their total spaces of the foliations mapping leaves to leaves diffeomorphically. 
\end{remark}

\begin{proof}
The result is proven in \cite{Hector} for the case of cylinder leaves and in \cite{Herman} for the case of plane leaves.
\end{proof}

\begin{theorem}
\label{thm:CauchyPairsClosed}
Let $(\fre,\Theta)$ be a Cauchy pair on an oriented closed three-manifold $\Sigma$ with associated foliation $\cF_{\fre}\subset \Sigma$ and Riemannian metric $h_{\fre}$. Then, one and only one of the following cases occur:

\begin{enumerate}[leftmargin=*]
	\item $\cF_{\fre}\subset \Sigma$ is a foliation by plane leaves and there exists an isometry:
	\begin{equation*}
	(\Sigma,h_{\fre}) =(\mathbb{R}^2\times S^1, \dd x_1\otimes \dd x_1 +\dd x_2\otimes \dd x_2 +  e^{2\fru}\dd \theta\otimes \dd \theta )/\langle \rho_a(1) , \rho_b(1)\rangle\, ,
	\end{equation*}
	
	\noindent
	where $ \rho_a(1) , \rho_b(1) \in \Diff(S^1)$ are rotations of rationally independent constant irrational angle, respectively, and $\mathfrak{u}\in C^{\infty}(\mathbb{R}^2)$ is a function depending only on $x_1$ and $x_2$. In particular, $\Sigma$ is diffeomorphic to $T^3$ and $\cF_{\fre}$ is isomorphic to the foliation $\cF_{\rho}$ described above.
	
	\item $\cF_{\fre}\subset \Sigma$ is a foliation by cylinder leaves and there exists an isometry:
	\begin{equation*}
	(\Sigma,h_{\fre}) =(\mathbb{R}^2\times S^1, \dd x_1\otimes \dd x_1 +\dd x_2\otimes \dd x_2 +  e^{2\fru}\dd \theta\otimes \dd \theta )/\langle \frf\rangle\, ,
	\end{equation*}
	
	\noindent
	where $ \mathfrak{f} \in \Diff(T^2\times \mathbb{R})$ is as prescribed in \eqref{eq:frfdefinitioncyl} and \eqref{eq:frfsl2z} and $\mathfrak{u}\in C^{\infty}(\mathbb{R}^2)$ is a function depending only on $x_1$ and $x_2$.  In particular, $\cF_{\fre}$ is isomorphic to the foliation $\cF_{\frf}$ described above.
	
	\item $\cF_{\fre}\subset \Sigma$ is a foliation by torus leaves and $(\Sigma,h_{\fre})$ is a conformal Riemannian submersion over $S^1$ with flat fibers and whose conformal factor is determined, modulo constant multiplicative factors, by:
	\begin{equation*}
	\Theta(e_u) = -\dd \frf\, .
	\end{equation*}
	
	\noindent
	 In particular, $\Sigma$ is diffeomorphic to a torus suspension by an element $\mathfrak{t}\in \Sl(2,\mathbb{Z})$. 
\end{enumerate}
\end{theorem}

\begin{proof}
We prove the statement point by point.

\begin{enumerate}[leftmargin=*]
	\item Let $(\fre,\Theta)$ be a parallel Cauchy pair with associated foliation $\cF_{\fre}\subset \Sigma$ by planes. Then, and as explained above, $\Sigma$ is diffeomorphic to $T^3$ (any compact connected 3-manifold with a foliation by planes is diffeomorphic to $T^3$), $(\Sigma,h_{\fre})$ is covered by $S^1\times\mathbb{R}^2$ and $\cF_{\fre}$ lifts to the plane foliation of $S^1\times\mathbb{R}^2$ whose leaves are embedded planes $\left\{\theta\right\}\times\mathbb{R}^2\subset S^1\times \mathbb{R}^2$, $\theta\in S^1$. Hence, the lift of $h_{\fre}$ to $S^1\times\mathbb{R}^2$ reads:
	\begin{equation*}
	(S^1\times \mathbb{R}^2 , \hat{h}_{\fre} = e^{2\fru}\dd \theta \otimes \dd \theta +  h_{\theta})\, ,
	\end{equation*} 
	
	\noindent
	where $\fru$ is a function on $S^1\times \mathbb{R}^2$, $\theta$ is an angular coordinate on $S^1$ and $h_{\theta}$ is a family of flat metrics on $\mathbb{R}^2$ parametrized by $\theta\in S^1$. Consequently, $(\Sigma,h_{\fre})$ has the following isometry type:
	\begin{equation*}
	(\Sigma,h_{\fre}) =(S^1\times \mathbb{R}^2, e^{2\fru}\dd \theta \otimes \dd \theta +  h_{\theta} )/\langle \rho_a(1) , \rho_b(1)\rangle\, ,
    \end{equation*}
	
	\noindent
	For the metric $e^{2\fru}\dd \theta \otimes \dd \theta +  h_{\theta}$ to descend to $\Sigma$ through the previous quotient we must have:
	\begin{equation*}
	\rho_a(1)^{\ast}( e^{2\fru}\dd \theta\otimes \dd \theta +  h_{\theta}) =  e^{2\fru}\dd \theta\otimes \dd \theta +  h_{\theta}\, , \, \rho_b(1)^{\ast}( e^{2\fru}\dd \theta \otimes \dd \theta +  h_{\theta}) =  e^{2\fru}\dd \theta \otimes \dd \theta +  h_{\theta}\, ,
	\end{equation*}
	
	\noindent
	which immediately implies:
	\begin{equation*}
	\fru \circ  \rho_o(1)= \fru\, ,  \quad h_{\theta\circ \rho_o(1)}=h_{\theta}\, ,
	\end{equation*}
	
	\noindent
	for $o = a,b$. Since $\langle \rho_a(1) , \rho_b(1)\rangle$ generates a dense subgroup (recall that the action of any diffeomorphism $\chi:S^1 \rightarrow S^1$ with constant irrational rotation number has dense orbits) of $S^1$ this implies in turn that $h_{\theta}$ and $\mathfrak{u}$ are constant along $S^1$.
	
	\item Let $(\fre,\Theta)$ be a parallel Cauchy pair with associated foliation $\cF_{\fre}\subset \Sigma$ by cylinder leaves. Then, and as explained above,  $(\Sigma,h_{\fre})$ is covered by $T^2\times \mathbb{R}$ and $\cF_{\fre}$ lifts to the cylinder foliation of $T^2\times\mathbb{R}$ whose leaves are the embedded cylinders $\mathbb{R}\times S^1\times \left\{\theta\right\}\subset T^2\times \mathbb{R}$, $\theta\in S^1$. Hence, the lift of $h_{\fre}$ to $S^1\times\mathbb{R}^2$ is given by:
	\begin{equation*}
		( S^1\times S^1\times \mathbb{R} , \hat{h}_{\fre} = e^{2\fru}\dd \theta_1 \otimes \dd \theta_1 +  h_{\theta_1})\, ,
	\end{equation*} 
	
	\noindent
	where $\fru$ is a function on $S^1\times S^1 \times\mathbb{R}$, $(\theta_1 , \theta_2)$ are angular coordinates on $S^1\times S^1$ and $h_{\theta_1}$ is a family of flat metrics on $S^1\times\mathbb{R}$ parametrized by $\theta_1\in S^1$. Then:
	\begin{equation*}
		(\Sigma,h_{\fre}) =(S^1\times S^1\times \mathbb{R}, e^{2\fru}\dd \theta_1 \otimes \dd \theta_1 +  h_{\theta_1} )/\langle \frf\rangle\, ,
	\end{equation*}
	
	\noindent
	For the metric $e^{2\fru}\dd \theta \otimes \dd \theta +  h_{\theta_1}$ to descend to $\Sigma$ the group we must have:
	\begin{equation*}
		\frf^{\ast}( e^{2\fru}\dd \theta_1 \otimes \dd \theta_1 +  h_{\theta_1}) =  e^{2\fru}\dd \theta_1 \otimes \dd \theta_1  +  h_{\theta_1}\, ,  
	\end{equation*}
	
	\noindent
	which, since the rotation number of $\frf$ is irrational, immediately implies, as in the previous case, that neither $h_{\theta_1}$ nor $\mathfrak{u}$ depend on $\theta_1$.

	\item Let $(\fre,\Theta)$ be a parallel Cauchy pair with associated foliation $\cF_{\fre}\subset \Sigma$ by torus leaves. Since $\cF_{\fre}$ has trivial holonomy and $\Sigma$ is connected and compact,\cite[Corollary 8.6]{Sharpe} implies that $\cF_{\fre}$ arises as the fibers of a fibration $\pi\colon \Sigma \to S^1$ and:
	\begin{equation*}
		T\Sigma = H\oplus V\, ,
	\end{equation*}
	
	\noindent
	where $V := \ker(\dd\pi)$ and $H$ is spanned by $e^{\sharp}_u$. In particular, the vertical bundle $V$ is spanned by $e^{\sharp}_l$ and $e^{\sharp}_n$, so the fibers of $\pi$ are flat and we obtain a conformal submersion over $S^1$. The fact that the conformal factor $e^{\mathfrak{f}}$ is as described in the statement follows from the first equation of the parallel Cauchy differential system, namely:
	\begin{equation*}
	\dd e_u  =  \dd \mathfrak{f}\wedge e_u\, ,
	\end{equation*} 
	
	\noindent
	which implies $\dd (e^{\mathfrak{f}}e_u) = 0$. Hence $e^{\mathfrak{f}}e_u$ is locally the exterior derivative of a coordinate $\hat{x}$ and the horizontal metric is locally $\dd \hat{x}\otimes \dd \hat{x}$. 
\end{enumerate}
\end{proof}


\section{Left-invariant parallel Cauchy pairs on Lie groups}
\label{sec:leftinvariant}


In this section we investigate left-invariant parallel Cauchy pairs on connected and simply connected three-dimensional Lie groups. In order to do this, we will exploit the classification of connected and simply connected three-dimensional Riemannian Lie groups developed in \cite{Milnor}, together with the fact that every left-invariant Cauchy pair $(\fre,\Theta)$ defines a left-invariant metric $h_{\fre}$.

Let $(\fre,\Theta) \in \Conf(\Sigma)$ be a left-invariant Cauchy pair on a three-dimensional connected and simply connected Lie group $\Sigma =\G$, that is, $\fre$ is a left-invariant coframe and $\Theta$ is a left-invariant shape operator on $\G$. Write:
\begin{equation*}
\Theta = \sum_{a,b}\Theta_{ab}\, e_a\otimes e_b\, , \qquad \Theta_{ab}\in \mathbb{R}\, , \qquad a, b = u,l,n\, .
\end{equation*}

\noindent
in terms of the left-invariant Cauchy coframe $\fre = (e_u , e_l , e_n)$. Using the previous expression for $\Theta$, the Cauchy differential system \eqref{eq:exder} evaluated on $(\fre,\Theta)$ is equivalent to:
\begin{equation}
\label{eq:leftinvcauchydifferential}
\dd e_u = (\Theta_{ul} e_l + \Theta_{un} e_n)\wedge e_u\, , \quad \dd e_l = (\Theta_{ll} e_l + \Theta_{ln} e_n)\wedge e_u\, , \quad \dd e_n =(\Theta_{nl} e_l + \Theta_{nn} e_n)\wedge e_u\, .
\end{equation}

\noindent
Taking the exterior derivative of the previous equations, we obtain the corresponding \emph{integrability conditions}:
\begin{equation}
\label{eq:integrabilityleft}
\tll\tun-\tln\tul=0\, , \qquad \tln\tun-\tnn\tul=0\, .
\end{equation}

\noindent
For further reference, we define the following quantities:
\begin{equation*}
T :=\tll+\tnn\, , \qquad \Delta := \tll\tnn-\tln^2\, ,
\end{equation*}

\noindent
which respectively correspond to the trace and determinant of $\Theta$ restricted to the distribution defined by the kernel of $e_u$.

\begin{proposition}
\label{prop:closedtheta}
A left invariant Cauchy pair $(\fre,\Theta)$ satisfies the cohomological condition $ [\Theta(e_u)]=0$ if and only if:
\begin{equation*}
(\tul^2+\tun^2) \mathrm{Tr}_{\fre}(\Theta)=0\, .
\end{equation*}
\end{proposition}

\begin{proof}
Since $\Sigma$ is by assumption simply-connected we have $H^1(\Sigma)=0$ and it suffices to prove that $\Theta(e_u)$ is closed. We impose:
\begin{equation*}
\dd \Theta(e_u)=\tuu \dd e_u+ \tul \dd e_l+ \tun \dd e_n=0\, .
\end{equation*}

\noindent
Using the parallel Cauchy differential system \ref{eq:leftinvcauchydifferential}, the previous condition is equivalent to the following equations:
\begin{equation*}
\tuu\tul+\tul\tll+\tun \tln=0\, , \quad \tuu\tun+\tul \tln+\tun \tnn=0\,,
\end{equation*}

\noindent
which, upon the use of the integrability condition \eqref{eq:integrabilityleft} of $(\fre,\Theta)$, are in turn equivalent to:
\begin{equation*}
\tul \mathrm{Tr}_{\fre}(\Theta) = 0\, , \qquad \tun \mathrm{Tr}_{\fre}(\Theta) = 0\, .
\end{equation*}

\noindent
These equations are satisfied if and only if $\tul=\tun=0$ or $\mathrm{Tr}_{\fre}(\Theta)=0$ (or both) hold. 
\end{proof}

\noindent
We consider now the case in which $\G$ is unimodular.

\begin{lemma}
Let $(\fre,\Theta)\in \Sol(\G)$ be a parallel Cauchy pair. Then, the simply connected three-dimensional group $\G$ is unimodular if and only if:
\begin{equation}
\label{eq:unimodularitycond}
T =\tll+\tnn=0\, , \qquad \tun=\tul=0\, .
\end{equation}
\end{lemma}

\begin{proof}
A Lie group $\G$ is unimodular if and only if the adjoint map of the associated Lie algebra has vanishing trace. Since the parallel Cauchy coframe $\fre = (e_u,e_l,e_n)$ is left-invariant, unimodularity of $\G$ is equivalent to:
\begin{equation*}
\dd e_l(e_u^\sharp,e_l^\sharp)+\dd e_n(e_u^\sharp,e_n^\sharp)=0\, , \quad \dd e_u(e_l^\sharp,e_u^\sharp)+\dd e_n(e_l^\sharp,e_n^\sharp)=0\, , \quad \dd e_u(e_n^\sharp,e_u^\sharp)+\dd e_l(e_n^\sharp,e_l^\sharp)=0\, ,
\end{equation*}
	
\noindent
which in turn is equivalent to:
\begin{equation*}
\tll+\tnn=0\, , \qquad \tul=0\, , \qquad \tun=0\, ,
\end{equation*}
	
\noindent
upon the use of the parallel Cauchy differential system \eqref{eq:leftinvcauchydifferential}.
\end{proof}

\begin{proposition}
\label{prop:cauchyuni}
Let $(\fre,\Theta)$ be a left invariant Cauchy pair on an unimodular Lie group $\G$. Then, one and only one of the following holds:
\begin{itemize}[leftmargin=*]
\item $\Delta=0$ and $(\G,h_{\fre})$ is isometric to the additive abelian Lie group $\mathbb{R}^3$ equipped with its standard invariant flat Riemannian metric.
\item $\Delta \neq 0$ and $\Sigma$ is isometric to the group $\mathrm{E}(1,1)$ of rigid motions of two-dimensional Minkowski space equipped with a left-invariant Riemannian metric. 
\end{itemize}
\end{proposition}

\begin{proof}
We distinguish between the cases $\Delta=0$ and $\Delta \neq 0$. 

\begin{itemize}[leftmargin=*]
\item $\Delta=0$. Since $\tll\tnn=\tln^2$ and we have $\tll+\tnn=0$ by unimodularity, we obtain that $\tll=\tnn=\tln=0$. Also, again by unimodularity, $\tul=\tun=0$, so we conclude that $\dd \fre = 0$ and $\Sigma$ is isomorphic to the abelian Lie group $\mathbb{R}^3$.

\item $\Delta \neq 0$. By unimodularity, see equation \eqref{eq:unimodularitycond}, we have $\tll=-\tnn$ and hence $\Delta < 0$. The exterior derivative of the Cauchy coframe $\fre$ can be then written as follows:
\begin{equation*}
\dd e_u=0\, , \quad  \dd e_l=(\tll e_l +\tln e_n) \wedge e_u \, , \quad \dd e_n=(\tln e_l -\tll e_n) \wedge e_u\,.
\end{equation*}

\noindent
If $\tll = 0$ the previous equations reduce to:
\begin{equation*}
\dd e_u=0\, , \quad  \dd e_l= \tln e_n \wedge e_u \, , \quad \dd e_n=\tln e_l \wedge e_u\, .
\end{equation*}

\noindent
Since $\Delta < 0$, we have $\tln \neq 0$ and after rescaling $e_u$ by $\tln$ we obtain:
\begin{equation*}
\dd e^{\prime}_u=0\, , \quad  \dd e_l=  e_n \wedge e^{\prime}_u \, , \quad \dd e_n=  e_l \wedge e^{\prime}_u\, .
\end{equation*}

\noindent
Comparing with the classification of unimodular Riemannian Lie groups \cite{Milnor}, see also Appendix A of \cite{Freibert} for a concise summary, existence of such left-invariant coframe implies that $\G$ is isomorphic to the Lie group $\mathrm{E}(1,1)$. If $\tll \neq 0$ we consider following change of coframes:
\begin{equation*}
\begin{pmatrix}
e_1 \\ e_2 \\e_3
\end{pmatrix}= \begin{pmatrix}
\cos \beta & -\sin \beta & 0 \\
\sin \beta & \cos \beta & 0 \\ 
0 & 0 & \sqrt{\vert \Delta \vert} \\
\end{pmatrix}\begin{pmatrix}
e_l \\ e_n \\ e_u \\
\end{pmatrix}\, ,
\end{equation*}

\noindent
where:
\begin{equation*}
\sin \beta=\frac{\sqrt{2}}{2} \sqrt{1-\frac{\tln}{\sqrt{\vert \Delta\vert}}}\, , \quad \cos \beta=\frac{\sqrt{2}}{2}\frac{\tll}{\sqrt{\vert\Delta\vert}} \frac{1}{\sqrt{1-\frac{\tln}{\sqrt{\vert\Delta\vert}}}}\, .
\end{equation*}

\noindent
The exterior derivative of the transformed coframe $(e_1,e_2,e_3)$ reads:
\begin{equation*}
\dd e_1=e_2 \wedge e_3 \, , \quad \dd e_2=e_1\wedge e_3\, , \quad \dd e_3=0\,.
\end{equation*}

\noindent
By the classification of unimodular Riemannian Lie groups \cite{Milnor}, existence of such left-invariant coframe implies that $\G$ is again isomorphic to the Lie group $\mathrm{E}(1,1)$, and hence we conclude. 
\end{itemize}
\end{proof}

\noindent
We consider now the case in which $\G$ is non-unimodular.

\begin{proposition}
\label{prop:Cauchynonuni}
Let $(\fre,\Theta)$ be a left invariant Cauchy pair on a non-unimodular Lie group $\G$. Then, one and only one of the following holds:
\begin{itemize}[leftmargin=*]
\item $\Delta=0$ and $(\G , h_{\fre})$ is isometric to the Lie group $\tau_2 \oplus  \mathbb{R}$ equipped with a left-invariant Riemannian metric.
\item $\Delta \neq 0$ and $(\G , h_{\fre})$ is isometric to $\tau_{3,\mu}$ equipped with a left-invariant Riemannian metric, where $\mu$ is given by one of the following possibilities:
\begin{enumerate}
\item If $\Theta(e_l,e_n) \neq 0$, by:
\begin{equation*}
\mu = \frac{T-\text{\emph{sign}}(T)\sqrt{T^2-4\Delta}}{T+\text{\emph{sign}}(T)\sqrt{T^2-4\Delta}}\, .
\end{equation*}
\item If $\Theta(e_l,e_n)=0$ and $\vert \Theta(e_l,e_l) \vert \geq \vert \Theta(e_n,e_n) \vert $, by:
\begin{equation*}
\mu=\frac{\Theta(e_n,e_n)}{\Theta(e_l,e_l)}\, .
\end{equation*}
\item If $\Theta(e_l,e_n)=0$ and $\vert \Theta(e_n,e_n) \vert \geq \vert \Theta(e_l,e_l) \vert $, by: 
\begin{equation*}
\mu=\frac{\Theta(e_l,e_l)}{\Theta(e_n,e_n)}\, .
\end{equation*}
\end{enumerate}

\noindent
Recall that the possible values of $\mu$ satisfy $-1<\mu \leq 1$, $\mu \neq 0$.
\end{itemize}
\end{proposition}

\begin{proof}
We distinguish between the cases $\Delta=0$ and $\Delta \neq 0$. 
\begin{itemize}[leftmargin=*]
\item $\Delta=0$. Assume first that $T=\tll+\tnn=0$. Conditions $T=0$ and $\Delta=0$ can hold simultaneously if and only if $\tll=\tnn=\tln=0$. Hence:
\begin{equation*}
\dd e_u=\tul e_l \wedge e_u+ \tun e_n \wedge e_u \, , \quad \dd e_l = 0\, , \quad  \dd e_n=0\, , \quad \Theta_{uu} = 0\, ,
\end{equation*}

\noindent
where the last equation is equivalent to the one-form $\Theta(e_u)$ being exact. Since the coefficients $\tul$ and $\tun$ cannot simultaneously vanish (otherwise $\G$ would be unimodular) defining $e_1=e_u$, $e_2=\tun e_l -\tul e_n $, $e_3=\tul e_l+\tun e_n$ we conclude that $\G$ is isomorphic to $\tau_2 \oplus \mathbb{R}$. 

\noindent
If $T \neq 0$, then either $\tll \neq 0$ or $\tnn \neq 0$ or both are non-vanishing. Assume $\tll \neq 0$ (completely analogue results hold if we consider $\tnn \neq 0$). In this case, the integrability conditions \eqref{eq:integrabilityleft} imply:
\begin{equation*}
\tun=\frac{\tln}{\tll} \tul\, .
\end{equation*}

\noindent
This equation, together with condition $\Delta = 0$, implies:
\begin{equation*}
\dd e_u= \tul( e_l+\frac{\tln}{\tll} e_n) \wedge e_u \, , \quad \dd e_l= \tll (e_l+\frac{\tln}{\tll}e_n ) \wedge e_u \, , \quad \dd e_n=\tln(e_l+\frac{\tln}{\tll} e_n) \wedge e_u\, ,
\end{equation*}

\noindent
which must be considered together with equation $(\Theta_{ul}^2 + \Theta_{un}^2) \mathrm{Tr}_{\fre}(\Theta) = 0$ to guarantee that $\Theta(e_u)$ is closed. We distinguish the following possibilities:

\begin{enumerate}[leftmargin=*]
\item $\tul=\tln=0$. In this case, it can be easily seen that $\G$ is isomorphic to $\tau_2 \oplus \mathbb{R}$.

\item $\tul=0$ and $\tln \neq 0$. In this case, we obtain:
\begin{equation*}
\dd e_u= 0 \, , \quad \dd e_l= \tll (e_l+\frac{\tln}{\tll}e_n ) \wedge e_u \, , \quad \dd e_n=\tln(e_l+\frac{\tln}{\tll} e_n) \wedge e_u\, .
\end{equation*}

\noindent
Defining $e_1 := e_l+\frac{\tln}{\tll}e_n$, $e_2 := e_l-\frac{\tll}{\tln}e_n$ and $e_3 :=T e_u$, we obtain:
\begin{equation*}
\dd e_1=  e_1 \wedge e_3 \, , \quad \dd e_2=\dd e_3=0\, ,
\end{equation*}

\noindent
Hence $\G$ is isomorphic to $\tau_2 \oplus \mathbb{R}$. 

\item $\tln=0$, but $\tul \neq 0$. In this case, we obtain:
\begin{equation*}
\dd e_u= \tul e_l \wedge e_u \, , \quad \dd e_l= \tll e_l \wedge e_u \, , \quad \dd e_n=0\, .
\end{equation*}

\noindent
Defining $e_1 := e_l+\frac{\tll}{\tul} e_u$, $e_2 := e_l-\frac{\tll}{\tul} e_u$ and $e_3 := e_n$, we conclude that $\G$ is isomorphic to $\tau_2\oplus \mathbb{R}$ once we impose $\tuu=-T$ in order to satisfy $[\Theta(e_u)]=0$.
 
\item $\tln\neq 0$ and $\tul \neq 0$. Define $e_2 := e_l + \frac{\tln}{\tll} e_n $ and $e_3 := e_l-\frac{\tll}{\tln}e_n$. We obtain:
\begin{equation*}
\dd e_u=\tul e_2 \wedge e_u \, , \quad \dd e_2= T e_2 \wedge e_u \, , \quad \dd e_3=0\,. 
\end{equation*}

\noindent
We  redefine $\tilde{e}_2=e_2-\frac{T}{\tul}e_u$ and $e_1=e_u$, we finally obtain:
\begin{equation*}
\dd e_1=-\tul e_1 \wedge \tilde{e}_2\, , \quad \dd \tilde{e}_2= 0 \, , \quad \dd e_3=0\, ,
\end{equation*}

\noindent
implying that $\G$ is isomorphic to $\tau_2 \oplus \mathbb{R}$ after imposing $\tuu=-T$ in order to guarantee $\Theta(e_u)$ to be closed. 
\end{enumerate}

\item $\Delta \neq 0 $. Since $\Delta \neq 0$, the only possible solution to the integrability conditions \eqref{eq:integrabilityleft} is $\tul=\tun=0$. Hence, non-unimodularity necessarily requires that $T = \tll+\tnn \neq 0$ and the parallel Cauchy differential system reduces to:
\begin{equation*}
\dd e_u=0 \, , \quad \dd e_l= (\tll e_l +\tln e_n) \wedge e_u \, , \quad \dd e_n=(\tln e_l+\tnn e_n) \wedge e_u\, .
\end{equation*}

\noindent
Assume $\tln \neq 0$ and define a global coframe $(e_1,e_2,e_3)$ in terms of the parallel Cauchy coframe $\fre$ as follows:
\begin{equation*}
\begin{pmatrix}
e_1 \\ \\  e_2 \\ \\ e_3
\end{pmatrix}= \begin{pmatrix}
1 & \dfrac{\lambda-\tll}{\tln} & 0 \\ \\
1 & \dfrac{\mu-\tll}{\tln} & 0 \\  \\
0 & 0 &    \lambda  \\
\end{pmatrix}\begin{pmatrix}
e_l \\ \\ e_n \\ \\ e_u \\
\end{pmatrix}\, ,
\end{equation*}

\noindent
where:
\begin{equation*}
\lambda= \frac{1}{2}(T+\text{sign}(T) \sqrt{T^2-4 \Delta})\, , \quad \mu= \frac{1}{2}(T-\text{sign}(T) \sqrt{T^2-4\Delta})\, .
\end{equation*}

\noindent
Note that $\lambda=\mu$ if and only if $\tll=\tnn$ and $\tln=0$, which is not possible since we are assuming $\tln \neq 0 $. The exterior derivative of $(e_1,e_2,e_3)$ can be shown to be given by:
\begin{equation*}
\dd e_1=e_1 \wedge e_3\, , \qquad \dd e_2= \tilde{\mu}  e_2 \wedge e_3 \, , \qquad \dd e_3=0\, ,
\end{equation*}

\noindent
where we defined $\tilde{\mu}=\frac{\mu}{\lambda}$. Note that $1>\vert \tilde{\mu} \vert >0$, since $\tln \neq 0$ and $\Delta \neq 0$. Hence, $\G$ is isomorphic to $\tau_{3, \tilde{\mu}}$. 

\noindent
If $\tln=0$, the exterior derivative of the Cauchy coframe $\fre$ reads:
\begin{equation*}
\dd e_u=0 \, , \quad \dd e_l= \tll e_l \wedge e_u \, , \quad \dd e_n=\tnn e_n \wedge e_u\, .
\end{equation*}

\noindent
Assume first that $\vert \tll \vert \geq \vert \tnn \vert $. Note that $\tll \neq -\tnn$ by non-unimodularity. By rescaling $e_u$, we obtain: 
\begin{equation*}
\dd e_u=0 \, , \quad \dd e_l= e_l \wedge e_u \, , \quad \dd e_n=\frac{\tnn}{\tll} e_n \wedge e_u\, .
\end{equation*}

\noindent
Since $1 \geq \frac{\tnn}{\tll}>-1$ and $\tnn \neq 0$ (otherwise $\Delta=0$), we conclude $\Sigma$ is isomorphic to $\tau_{3, \frac{\tnn}{\tll}}$. An analogous conclusion holds if $\vert \tnn \vert \geq \vert \tll \vert $.
\end{itemize}
\end{proof}

\begin{proposition}
\label{prop:Codazzicity}
The shape operator $\Theta$ of a parallel Cauchy pair $(\fre,\Theta)$ on $\G$ is Codazzi if and only if:
\begin{equation*}
C_a \eqdef 	e_u\otimes \Theta\circ \Theta(e_a) - \Theta(e_u)\otimes \Theta(e_a) - \delta_{ua}\Theta\circ \Theta + \Theta_{ua}\Theta = 0\,
\end{equation*}

\noindent
for every $a = u , l , n$.
\end{proposition}

\begin{proof}
We compute:
\begin{equation}
\label{eq:nablaThetaG}
\nabla^{\fre}_{e_a} \Theta = - \Theta(e_u)\otimes \Theta(e_a) - \Theta(e_a)\otimes \Theta(e_u) + \Theta\circ \Theta(e_a) \otimes e_u + e_u\otimes \Theta\circ \Theta(e_a)\, , 
\end{equation}
	
\noindent
Similarly:
\begin{equation*}
(\nabla^{\fre} \Theta)(e_a) = - \Theta(e_u,e_a) \Theta + \delta_{ua} \Theta\circ \Theta  + \Theta\circ \Theta(e_a)\otimes e_u - \Theta(e_a)\otimes \Theta(e_u)\, .
\end{equation*}

\noindent
Since $\Theta$ is Codazzi if and only if $\nabla^{\fre}_{e_a} \Theta = (\nabla^{\fre} \Theta)(e_a) $ for all $a=u,l,n$, matching the previous pair of equations we obtain:
\begin{equation*}
e_u\otimes \Theta\circ \Theta(e_a) - \Theta(e_u)\otimes \Theta(e_a) - \delta_{ua}\Theta\circ \Theta + \Theta_{ua}\Theta = 0\, .
\end{equation*}
\end{proof} 
\noindent
\begin{remark}
\label{remark:codazzicity}
It is not hard to see that: 
\begin{equation*}
C_a (e_b,e_d)=-C_b(e_a,e_d)\, ,
\end{equation*}
for every $a, b, c=u, l, n$. We will be use this identity momentarily.
\end{remark}

\begin{proposition}
\label{prop:consricciflat}
A parallel Cauchy pair $(\fre,\Theta)$ on $\G$ is constrained Ricci-flat if and only if:
\begin{equation*}
\Theta(e_u,e_u) \mathrm{Tr}_{\fre}\Theta = \vert \Theta \vert^2_{\fre}\, .
\end{equation*}
\end{proposition}

\begin{proof}
By Proposition \ref{prop:hamiltonianmomentum}, the Hamiltonian and momentum constraints for a Cauchy pair are equivalent. We consider the momentum constraint. We have $\dd \mathrm{Tr}_{h_{\fre}} (\Theta)=0$. Hence, by Lemma \ref{lemma:eumc} the constraint Ricci-flatness condition for $(\Theta,\fre)$ is equivalent to:
\begin{equation*}
\mathrm{div}_{\fre} (\Theta)(e_u)= 0\, .
\end{equation*}

\noindent
Using Equation \eqref{eq:nablaThetaG} we compute:
\begin{equation*}
	\mathrm{div}_{\fre} \Theta = \sum_a (\nabla^{\fre}_{e_a}\Theta)(e_a) = - \mathrm{Tr}_{\fre}(\Theta) \Theta(e_u) + \vert \Theta\vert_{\fre}^2 e_u\, ,
\end{equation*}

\noindent
and therefore we conclude.
\end{proof}
\noindent
\begin{lemma}
\label{lemma:whencodazzi}
The shape operator $\Theta$ of a parallel Cauchy pair $(\fre, \Theta)$ is Codazzi if and only if it satisfies one of the following conditions:
\begin{itemize}[leftmargin=*]
\item $\tul=\tun=\tln=0$, $\tll^2=\tll \tuu$, $\tnn^2=\tnn \tuu\,$.
\item $\Theta(e_u)=T e_u$,  $\Delta=0\,$.
\end{itemize}
 
\end{lemma}
\begin{proof}
Let $C_a \in \Gamma(T^* \G \otimes T^*\G)$ denote the tensor defined in Proposition \ref{prop:Codazzicity}. Remark \ref{remark:codazzicity} states that the only non-trivial and independent components are those corresponding to $C_a(e_u,e_l)$, $C_a(e_u,e_n)$ and $C(e_l,e_n)$. Imposing these components to vanish we obtain:
\begin{eqnarray}
\label{eq:proofcodazzi}
 & 2 \tul^2+\tll^2+\tln^2-\tuu\tll=0 \, , \quad 2\tun^2+\tnn^2+\tln^2-\tuu \tnn=0\, , \nonumber\\
 & 2\tul \tun + \tnn\tln+ \tln \tll-\tuu \tln=0\, , 
\end{eqnarray}

\noindent
In order to solve them we impose the cohomological condition as stated in Proposition \ref{prop:closedtheta}. Since the cohomological condition is satisfied if either $\tul=\tun=0$ or $\tuu=-T$, we distinguish between these two cases:
\begin{itemize}[leftmargin=*]
\item $\tul=\tun=0$. Let us split this case into two subcategories:
\begin{itemize}[leftmargin=*]
\item $\tln=0$. One notices that the equations reduce directly to $\tuu \tnn=\tnn^2$ and $\tuu\tll=\tll^2$. 
\item $\tln \neq 0$. In such a case, from the last equation of \eqref{eq:proofcodazzi} one finds $\tuu=T$ and, upon substitution in the remaining equations they become linearly dependent and equivalent to the condition $\Delta=0$. 
\end{itemize}
\item $\tuu=-T$. In such a case, by summing the first and the second equations of \eqref{eq:proofcodazzi} and performing explicitly the substitution $\tuu=-T$, we find
\begin{equation}
2\tul^2+2\tun^2+(\tll+\tnn)^2+2\tln^2+\tll^2+\tnn^2=0\,.
\end{equation}
This implies $\tul=\tun=\tll=\tnn=\tln=\tuu=0$, which brings us to the previous bullet-point. 
\end{itemize}
\end{proof}

\noindent
We elaborate now on the results of the previous discussion in order to obtain a full classification result about left-invariant parallel Cauchy pairs $(\fre,\Theta)$ on connected and simply connected three-dimensional Lie groups, characterizing those which are in addition Codazzi or constrained Ricci-flat. Collecting all results from Propositions \ref{prop:cauchyuni} and \ref{prop:Cauchynonuni} and bearing in mind Proposition \ref{prop:consricciflat} and Lemma \ref{lemma:whencodazzi}, we obtain the following result.

\begin{theorem}
A connected and simply-connected Lie group $\G$ admits left-invariant parallel Cauchy pairs (respectively constrained Ricci-flat parallel Cauchy pairs or a Codazzi parallel Cauchy pairs) if and only if $\G$ is isomorphic to one of the Lie groups listed in the Table below. If that is the case, a left-invariant shape operator $\Theta$ belongs to a Cauchy pair $(\fre,\Theta)$ for certain left-invariant coframe $\fre$ if and only if $\Theta$ is of the form listed below when written in terms of $\fre = (e_u,e_l,e_n)$:

\renewcommand{\arraystretch}{1.5}
\begin{center}
\begin{tabular}{|  p{1cm}| p{5cm} | p{3.7cm} | p{3cm} |}
\hline
$\mathrm{G}$ & \emph{Cauchy parallel pair} & \emph{Constrained Ricci-flat} & \emph{Codazzi} \\ \hline
$\mathbb{R}^3$ & $\Theta=\Theta_{uu} e_u \otimes e_u$ &  $\Theta=\Theta_{uu} e_u \otimes e_u$ & $\Theta=\Theta_{uu} e_u \otimes e_u$  \\ \hline
\multirow{2}*{$ \mathrm{E}(1,1)$} & $\Theta=\Theta_{uu} e_u \otimes e_u+ \Theta_{i j} e_i \otimes e_j$ & \multirow{2}*{\emph{Not allowed}} & \multirow{2}*{\emph{Not allowed}} \\ & $i,j=l,n,\, \quad \Theta_{ll}=-\Theta_{nn}$& & \\ \hline \multirow{11}*{$\tau_2 \oplus \mathbb{R}$} & $\Theta=(\tul e_l +\tun e_n) \odot e_u$  & \multirow{2}*{\emph{Not allowed}} & \multirow{2}*{\emph{Not allowed}}\\ & $\tul^2+\tun^2 \neq 0$ & &  \\ \cline{2-4} & $\Theta=\Theta_{uu} e_u \otimes e_u+ \Theta_{i j} e_i \otimes e_j$ &  $\Theta=T e_u \otimes e_u+ \Theta_{i j} e_i \otimes e_j$ &  $\Theta=T e_u \otimes e_u+ \Theta_{i j} e_i \otimes e_j$ \\ & $\begin{aligned} &i,j=l,n,\, \\ &T \neq 0\, , \Delta=0 \end{aligned}$ & $\begin{aligned} &i,j=l,n,\, \\ &T \neq 0\, , \Delta=0 \end{aligned}$& $\begin{aligned} &i,j=l,n,\,  \\ &T \neq 0\, , \Delta=0 \end{aligned}$\\ \cline{2-4} & $\Theta=-T e_u \otimes e_u+\tul e_u \odot e_l + \tll e_l \otimes e_l\, , \quad \tul, \tnn \neq 0$ & \emph{Not allowed} & \emph{Not allowed} \\ \cline{2-4} & $\Theta=-T e_u \otimes e_u+\tun e_u \odot e_n + \tnn e_n \otimes e_n\, , \quad \tun, \tll \neq 0$ & \emph{Not allowed} & \emph{Not allowed} \\\cline{2-4} & $\Theta=-T e_u\otimes e_u + \tul e_u \odot e_l+ \tun e_u \odot e_n+\Theta_{i j} e_i \otimes e_j\,$ & \multirow{3}*{\emph{Not allowed}}&  \multirow{3}*{\emph{Not allowed}} \\  & $\begin{aligned}&i,j=l,n,\, \, \tln( \tul^2+\tun^2) \neq 0\, , \\&\tnn=\frac{\tun}{\tul} \tln\, , \tll=\frac{\tul}{\tun} \tln \end{aligned}$ & & \\ & & & \\ \hline \multirow{2}*{$ \tau_{3,\mu}$} & $\Theta=\Theta_{uu} e_u \otimes e_u+ \Theta_{i j} e_i \otimes e_j$ & $\Theta=\left (\frac{T^2-2\Delta}{T} \right ) e_u \otimes e_u+ \Theta_{i j} e_i \otimes e_j$  & \multirow{2}*{\emph{Not allowed}} \\ & $i,j=l,n,\, \quad T, \Delta \neq 0$& $i,j=l,n,\, \quad T, \Delta \neq 0$ & \\ \hline

\end{tabular}

\end{center}
Regarding the case in which $\G \simeq \tau_{3,\mu}$:
\begin{itemize}[leftmargin=*]

\item If $\tln \neq 0$, then
\begin{equation*}
\mu = \frac{T-\text{\emph{sign}}(T)\sqrt{T^2-4\Delta}}{T+\text{\emph{sign}}(T)\sqrt{T^2-4\Delta}}\, .
\end{equation*}
\item If $\tln=0$ and $\vert\tll\vert \geq \vert \tnn \vert $, then
\begin{equation*}
\mu=\frac{\tnn}{\tll}\,.
\end{equation*}
\item If $\tln=0$ and $\vert \tnn \vert \geq \vert \tll \vert $, then
\begin{equation*}
\mu=\frac{\tll}{\tnn}\,.
\end{equation*}
\end{itemize}
 
\renewcommand{\arraystretch}{1}
\label{thm:allcauchygroups}
\end{theorem}

\noindent
We hope that the previous theorem, together with Lemma \ref{lemma:globhyperspinor}, can serve as the basis of a formulation of the parallel spinor flow on Lie groups and homogeneous three-manifolds, in the spirit of \cite{Panagiotis}.


\section{Comoving parallel spinor flows}
\label{sec:examplesflows}


In this section we consider a specific type of parallel spinor flow which admits a particularly neat geometric description,  with the goal of obtaining explicit \emph{time-dependent} Lorentzian four-manifolds  admitting parallel spinors. 


\subsection{Globally hyperbolic comoving spacetimes}


We consider a particular class of parallel spinor flows defined by imposing the condition $\lambda_t = 1$ for all $t\in \mathbb{R}$.

\begin{definition}
\label{def:comovingparallelspinorflow}
A parallel spinor flow $(\left\{ \lambda_t \right\}_{t\in \mathbb{R}} , \left\{ h_t \right\}_{t\in \mathbb{R}},\left\{ u^0_t\right\}_{t\in \mathbb{R}},\left\{ u^{\perp}_t\right\}_{t\in \mathbb{R}} ,\left\{ l^{\perp}_t\right\}_{t\in \mathbb{R}})$ is \emph{comoving} if $\lambda_t = 1$ for every $t\in \mathbb{R}$. 
\end{definition}

\noindent
A comoving parallel spinor flow on a manifold of the form $M = \mathbb{R}\times \Sigma$, where $\Sigma$ is an oriented three-manifold, will be always understood as a parallel spinor flow on $\Sigma$ with respect to the cartesian coordinate of $\mathbb{R}$.

\begin{definition}
\label{def:comovinggloballyhyper}
A four-dimensional space-time $(M,g)$ is a \emph{comoving globally hyperbolic space-time} if it is isometric to a model of the form:
\begin{equation*}
(M,g) = (\cI\times \Sigma , -\dd t\otimes \dd t + h_t)\, ,
\end{equation*}
	
\noindent
for a family $\left\{ h_t\right\}_{t\in\cI}$ of complete Riemannian metrics on $\Sigma$, where $\cI \subset \mathbb{R}$ is an interval.
\end{definition}

\noindent
A metric of the type $g=-\dd t^2 + h_t$ will be called a \emph{comoving globally hyperbolic}. 

\begin{remark}
The term \emph{comoving} is motivated by the fact that the local metric of a comoving observer in a cosmological background is of comoving globally hyperbolic type. In particular, the time factor of the metric is constant.
\end{remark}

\begin{theorem}
\label{thm:comovingparallelflow}
An oriented four-manifold $(M,g)$ admits a comoving parallel spinor flow if and only if $\Sigma $ admits a family:
\begin{equation*}
\left\{ \fre^t := (e_u^t , e_l^t, e_n^t) \colon \Sigma \to \mathrm{F}(\Sigma) \right\}_{t\in \mathbb{R}}\, ,
\end{equation*}
	
\noindent
of sections of the oriented frame bundle $\mathrm{F}(\Sigma)$ of $\Sigma$ satisfying the following system of partial differential equations:
\begin{eqnarray}
\label{eq:globhyperspinorcosmo}
\partial_t \fre^t +  \Theta_t(\fre^t) =  0\, , \quad \dd \fre^t = \Theta_t(\fre^t) \wedge e^t_u\, ,  \quad  [\Theta_t(e^t_u)]=0\in H^1(\Sigma,\mathbb{R})\, , \quad \partial_t\Theta_t(e^t_u) = 0\, ,
\end{eqnarray}

\noindent
where:
\begin{equation*}
\Theta_t = - \frac{1}{2}\partial_t (e_u^t\otimes e_u^t + e_l^t\otimes e_l^t + e_n^t\otimes e_n^t)\, .
\end{equation*} 

\noindent
If this is the case, the corresponding comoving globally hyperbolic metric is given by:
\begin{equation*}
	g = - \dd t\otimes \dd t + h_{\fre^t}\, , \qquad  h_{\fre^t} = e_u^t\otimes e_u^t + e_l^t\otimes e_l^t + e_n^t\otimes e_n^t\, .
\end{equation*} 
\end{theorem}

\begin{proof}
Consider a solution of equations \eqref{eq:globhyperspinorI}, \eqref{eq:globhyperspinorII} and \eqref{eq:restrictionul} of the form:
\begin{equation*}
(\left\{ \lambda_t = 1 \right\}_{t\in \mathbb{R}} , \left\{ h_t \right\}_{t\in \mathbb{R}},\left\{ u^0_t\right\}_{t\in \mathbb{R}},\left\{ u^{\perp}_t , l^{\perp}_t\right\}_{t\in \mathbb{R}})\, .
\end{equation*}
	
\noindent
To every such solution we canonically associate a family of sections $\left\{\fre^t\right\}_{t\in\mathbb{R}} \colon \Sigma \to \mathrm{F}(\Sigma)$ defined as follows:
\begin{equation*}
\fre^t = (e_u^t , e_l^t, e_n^t) \eqdef ( u^{\perp}_t /u^0_t , l^{\perp}_t , \ast_t (u^{\perp}_t \wedge l^{\perp}_t)/u^0_t)\, .
\end{equation*}
	
\noindent
where $\ast_t$ denotes the Hodge dual on $\Sigma$ with respect to $h_t$ and the fixed orientation. Clearly, the triple $\left\{(e_u^t ,e_l^t ,e_n^t )\right\}_{t\in \mathbb{R}}$ yields a family of orthonormal coframes on $\Sigma$ with respect to $\left\{ h_t\right\}_{t\in \mathbb{R}}$. Imposing $\lambda_t = 1$ for every $t\in\mathbb{R}$ in equations \eqref{eq:globhyperspinorI} and \eqref{eq:globhyperspinorII} we obtain:
\begin{eqnarray*}
& \partial_t u^{\perp}_t + \Theta_t(u^{\perp}_t) = 0\, , \qquad  \partial_t e^t_l  + \Theta_t(e^t_l) = 0\, ,   \\
& \nabla^{h_t} u^{\perp}_t + u^0_t \Theta_t = 0\, , \qquad \nabla^{h_t} e^t_l = \Theta_t(e^t_l)\otimes e^t_u\, .
\end{eqnarray*}
	
\noindent
Taking the time derivative of the constraint $(u_t^0)^2=\vert u_t^\perp \vert_{h_t}^2$, we conclude $\partial_t u_t^0=0$. Hence using the time independence of $u^0_t$ in the previous equations, we obtain:
\begin{eqnarray*}
&\frac{1}{u^0_t} \nabla^{h_t} u^{\perp}_t + \Theta_t = \nabla^{h_t} e^t_u + \dd\log(u^0_t)\otimes e^t_u + \Theta_t = \nabla^{h_t} e^t_u - \Theta_t(e^t_u)\otimes e^t_u + \Theta_t\, ,\\
&\partial_t e^t_u + \Theta_t(e^t_u) = 0\, .
\end{eqnarray*}
	
\noindent
Similar arguments yield the equation $\partial_t e_n^t+\Theta(e_n^t)=0$. On the other hand:
\begin{eqnarray*}
& \nabla^{h_t} e_n^t = \nabla^{h_t}  \ast_{h_t} (e^t_u\wedge e^t_l) =    \ast_{h_t} (\nabla^{h_t} e^t_u\wedge e^t_l) + \ast_{h_t} ( e^t_u\wedge \nabla^{h_t} e^t_l) \\
& = \ast_{h_t} (\nabla^{h_t} e^t_u\wedge e^t_l) = \Theta(e^t_n) \otimes e^t_u \, .
\end{eqnarray*}
\noindent
Hence, every parallel spinor flow of the form $(\left\{ \lambda_t = 1 \right\}_{t\in \mathbb{R}} , \left\{ h_t \right\}_{t\in \mathbb{R}},\left\{ u^0_t \right\}_{t\in \mathbb{R}},\left\{ u^{\perp}_t , l^{\perp}_t\right\}_{t\in \mathbb{R}})$ produces a canonical section of $\mathrm{F}(\Sigma)$ satisfying \eqref{eq:globhyperspinorcosmo}. Conversely, assume that $\left\{ \fre^t \right\}_{t\in \mathbb{R}}$ is a solution of \eqref{eq:globhyperspinorcosmo}. The third equation in \eqref{eq:globhyperspinorcosmo} guarantees that there exist a family of functions $\left\{ u^0_t\right\}_{t\in\mathbb{R}}$ such that:
\begin{equation*}
\dd\log(u^0_t) = -\Theta_t(e_u^t)\, .
\end{equation*}
	
\noindent
Plugging this equation into the first equation of \eqref{eq:globhyperspinorcosmo}, we obtain:
\begin{equation*}
\partial_t (u^0_t e^t_u) = \partial_t u^0_t e^t_u - u^0_t \Theta_t(e^t_u) =  \partial_t \log(u^0_t) u^0_t e^t_u -  \Theta_t(u^0_t e^t_u) =  -  \Theta_t(u^0_t e^t_u)\, ,
\end{equation*}
	
\noindent
which yields the first equation in \eqref{eq:globhyperspinorI}. The second equation in \eqref{eq:globhyperspinorI} follows similarly. Equations \eqref{eq:globhyperspinorII} and \eqref{eq:restrictionul} follow by interpreting the second equation in  \eqref{eq:globhyperspinorcosmo} as the first Cartan structure equation for the orthonormal frame $\fre^t$ with respect to the metric:
\begin{equation*}
h_{\fre^t} = e^t_u \otimes e^t_u + e^t_l \otimes e^t_l + e^t_n \otimes e^t_n\, ,
\end{equation*}
	
\noindent
and hence we conclude.
\end{proof}

\noindent
We will refer to equations \eqref{eq:globhyperspinorcosmo} as the \emph{comoving parallel-spinor flow equations}, and we will refer to its solutions as \emph{comoving parallel spinor flows}. The general investigation of comoving parallel spinor flows is beyond the scope of this article and will be considered in a separate publication. Instead, we consider two particular important cases  in detail. 


\subsection{A diagonal example on $\mathbb{R}^3$.}
\label{sec:diagonalexample}


Set $\Sigma = \mathbb{R}^3$ with Cartesian coordinates $(x,y,z)$ and consider comoving parallel spinor flows $\left\{ \fre^t\right\}_{t\in \cI}$ of the form:
\begin{equation*}
\fre^t = (\frf^t_u \dd x , \frf^t_l \dd y , \frf^t_n \dd z)\, ,
\end{equation*}

\noindent
for families of functions $\left\{ \frf^t_u\right\}_{t\in\mathbb{R}}$, $\left\{ \frf^t_l\right\}_{t\in\mathbb{R}}$ and $\left\{ \frf^t_n\right\}_{t\in\mathbb{R}}$ on $\mathbb{R}^3$. Hence:.
\begin{equation*}
h_{\fre^t} = (\frf^t_u)^2 \dd x \otimes \dd x + (\frf^t_l)^2 \dd y\otimes \dd y  + (\frf^t_n)^2 \dd z \otimes \dd z \, , \quad (\fre^t)^{\sharp} = (\frac{1}{\frf^t_u} \partial_x , \frac{1}{\frf^t_l} \partial_y , \frac{1}{\frf^t_n} \partial_z)\, ,
\end{equation*}

\noindent
With these provisos in mind, we compute:
\begin{equation}
\Theta_t = - (\frf^t_u \partial_t\frf^t_u\, \dd x\otimes \dd x + \frf^t_l \partial_t\frf^t_l\, \dd y\otimes \dd y + \frf^t_n \partial_t\frf^t_n\, \dd z\otimes \dd z )\, .
\label{eq:thetadiagonal}
\end{equation}

\noindent
Therefore, equation $\partial_t \fre^t  + \Theta_t(\fre^t) = 0$ is automatically satisfied. On the other hand, equations $[\Theta_t(e^t_u)] = 0$ and $\partial_t\Theta_t(e^t_u) = 0$ are equivalent to:
\begin{equation*}
\partial_t \dd \frf_u \wedge \dd x = 0\, , \qquad \partial_t^2 \frf_u = 0\, ,
\end{equation*}

\noindent
implying:
\begin{equation*}
\frf^t_u = \mathfrak{a}  + \mathfrak{b}\,  t\, ,
\end{equation*}

\noindent
where $\mathfrak{b} = \mathfrak{b}(x)$ is a function of the coordinate $x$ and $\mathfrak{a} = \mathfrak{a}(x,y,z)$ is a function of all coordinates of $\mathbb{R}^3$. Note that, in order to have a well defined comoving parallel spinor flow, we must impose the constraint:
\begin{equation*}
\frf^t_u(t,x,y,z) = \mathfrak{a}(x,y,z)  + \mathfrak{b}(x) \, t \neq 0\, ,	
\end{equation*}

\noindent
for every $t\in \cI$ and $(x,y,z)\in \mathbb{R}^3$, which translates into a constraint in the allowed domain of definition $\cI\subset \mathbb{R}$ of $t$. The only equations that remain to be solved for
\begin{equation*}
\fre^t = ((\mathfrak{a} + \mathfrak{b}\, t) \dd x , \frf^t_l \dd y , \frf^t_n \dd z)\, ,
\end{equation*}

\noindent
to be a comoving parallel spinor flow are $\dd \fre^t = \Theta(\fre^t)\wedge e^t_u$, which can be shown to be equivalent to:
\begin{equation*}
\dd \mathfrak{a}\wedge \dd x = 0\, , \qquad (\dd \frf^t_l - \partial_t \frf^t_l \frf^t_u \dd x )\wedge \dd y = 0\, , \qquad (\dd \frf^t_n - \partial_t \frf^t_n \frf^t_u \dd x )\wedge \dd z = 0\, . 
\end{equation*}

\noindent
These equations are in turn equivalent to:
\begin{equation}
\label{eq:equivalentcauchy}
\mathfrak{a} = \mathfrak{a}(x)\, , \quad \partial_x \frf^t_l = \frf^t_u  \partial_t \frf^t_l\, ,\quad \partial_z \frf^t_l = 0\, ,  \quad \partial_x \frf^t_n = \frf^t_u  \partial_t \frf^t_n\, ,\quad \partial_y \frf^t_n = 0\, ,
\end{equation}

\noindent
which do have explicit solutions, as we will show later in particular examples. On the other hand a direct computation shows that the Ricci curvature of the comoving globally hyperbolic Lorentzian metric $g=-\dd t\otimes \dd t + h_{\fre^t}$ associated to such $\fre^t$ vanishes if and only if the following condition holds:
\begin{equation}
\mathfrak{b} \left ( \frac{\partial_t \frf_l^t}{\frf_l}+\frac{\partial_t \frf_n^t}{\frf_n} \right)-\frac{\partial_t \partial _x \frf_l^t}{\frf_l}-\frac{\partial_t \partial _x \frf_n^t}{\frf_n}=0\, .
\label{eq:ricciflatdiagonal}
\end{equation}

\noindent
This condition will be explored in the examples below.

\begin{example}
Suppose that both $\mathfrak{a}$ and $\mathfrak{b}$ are constants, with $\mathfrak{b} \neq 0$. With this assumption, a general solution of equations \eqref{eq:equivalentcauchy} is of the form:
\begin{equation*}
\frf^t_l = \mathfrak{L}_l(x+\log\vert \mathfrak{a} +  \mathfrak{b}\, t\vert / \mathfrak{b},y)\, , \quad \frf^t_n = \mathfrak{L}_n(x+\log\vert \mathfrak{a} +  \mathfrak{b}\, t \vert /\mathfrak{b}, z)\, ,
\end{equation*}

\noindent
for nowhere vanishing smooth functions $\mathfrak{L}_l, \mathfrak{L}_n \in C^{\infty}(\mathbb{R}^2)$. The corresponding coframe $\fre^t$ reads:
\begin{equation*}
\fre^t = ((\mathfrak{a} + \mathfrak{b}t) \dd x , \mathfrak{L}_l(x+\log\vert \mathfrak{a} +  \mathfrak{b}\, t \vert / \mathfrak{b},y) \dd y ,  \mathfrak{L}_n(x+\log\vert \mathfrak{a} +  \mathfrak{b}\, t \vert / \mathfrak{b},z) \dd z)\, ,
\end{equation*}

\noindent
which is well-defined in the intervals $ t\in \cI_1 = (-\infty, -\frac{\mathfrak{a}}{\mathfrak{b}})$ or $ t\in \cI_2 = (-\frac{\mathfrak{a}}{\mathfrak{b}}, \infty)$. The metric associated to the previous global coframe is given by:
\begin{equation*}
g = - \dd t\otimes \dd t + (\mathfrak{a} + \mathfrak{b}\, t)^2 \dd x \otimes \dd x + \mathfrak{L}_l(x+\log\vert \mathfrak{a} +  \mathfrak{b}\, t \vert / \mathfrak{b},y)^2 \dd y\otimes \dd y  + \mathfrak{L}_n(x+\log\vert \mathfrak{a} +  \mathfrak{b}\, t \vert / \mathfrak{b},z)^2 \dd z \otimes \dd z \, ,
\end{equation*}

\noindent
which provides a large family of four-dimensional Lorentzian metrics admitting a parallel spinor. If the induced Riemannian spatial metric:
\begin{equation*}
h_{\fre^t} = (\mathfrak{a} + \mathfrak{b}\,  t)^2 \dd x \otimes \dd x + \mathfrak{L}_l(x+\log\vert \mathfrak{a} +  \mathfrak{b}\, t \vert / \mathfrak{b},y)^2 \dd y\otimes \dd y  + \mathfrak{L}_n(x+\log\vert \mathfrak{a} +  \mathfrak{b}\, t \vert / \mathfrak{b},z)^2 \dd z \otimes \dd z \, ,
\end{equation*}

\noindent
on $\left\{t\right\}\times \mathbb{R}^3 \subset \cI_i\times \mathbb{R}^3$ (for $i=1,2$) is complete for all $t\in \cI_{i}$ we obtain a family of comoving globally hyperbolic metrics on $\cI_{i}\times \mathbb{R}^3$. Equation \eqref{eq:ricciflatdiagonal}, implies now that $g$ is Ricci-flat if and only if:
\begin{equation*}
\frac{\mathfrak{b} \partial_\zeta \mathfrak{L}_l- \partial_\zeta \partial_\zeta \mathfrak{L}_l }{\mathfrak{L}_l}+\frac{\mathfrak{b} \partial_\zeta \mathfrak{L}_n- \partial_\zeta \partial_\zeta \mathfrak{L}_n }{\mathfrak{L}_n}=0\,,
\end{equation*}
where we have defined $\zeta(t,x):=x+\log\vert \mathfrak{a}+\mathfrak{b}\, t \vert /\mathfrak{b}$. This Ricci-flatness condition is satisfied if the functions $\mathfrak{L}_l(\zeta,y)$ and $\mathfrak{L}_n(\zeta,z)$ take the form:
\begin{equation*}
\mathfrak{L}_l(\zeta,y)=w_1 (y) e^{\mathfrak{b} \zeta}+w_2(y)\, , \quad \mathfrak{L}_n(\zeta,z)=w_3 (z) e^{\mathfrak{b} \zeta}+w_4(z)\, ,
\end{equation*}
where $w_1,w_2,w_3,w_4$ are arbitrary smooth functions. 

\end{example}

\begin{example}
Assume that $\mathfrak{a}$ is a possibly non-constant strictly positive function and $\mathfrak{b} = 0$. With this assumption, the general solution of equations \eqref{eq:equivalentcauchy} is of the form:
\begin{equation*}
	\frf^t_l = \mathfrak{L}_l \left (t+\int_0^x \mathfrak{a}(\tau)\dd\tau ,y \right )\, , \quad \frf^t_n = \mathfrak{L}_n \left (t+\int_0^x \mathfrak{a}(\tau)\dd\tau , z \right )\, ,
\end{equation*}

\noindent
for nowhere vanishing smooth functions $\mathfrak{L}_l, \mathfrak{L}_n \in C^{\infty}(\mathbb{R}^2)$. The corresponding coframe $\fre^t$ reads:
\begin{equation*}
	\fre^t = \left (\mathfrak{a}(x)   \dd x , \mathfrak{L}_l\left (t+\int_0^x \mathfrak{a}(\tau)\dd\tau  ,y \right ) \dd y ,  \mathfrak{L}_n\left (t+\int_0^x \mathfrak{a}(\tau)\dd\tau  , z \right ) \dd z \right )\, ,
\end{equation*}

\noindent
which is well-defined for $ t\in \cI = \mathbb{R}$. The metric associated to the previous global coframe is given, after a change and relabeling of coordinates, by the following expression:
\begin{equation*}
	g = - \dd t\otimes \dd t +   \dd x \otimes \dd x + \mathfrak{L}_l(t+x ,y)^2 \dd y\otimes \dd y  + \mathfrak{L}_n(t+x , z)^2 \dd z \otimes \dd z \, ,
\end{equation*}

\noindent
which provides a large family of four-dimensional Lorentzian metrics admitting a parallel spinor. If the induced Riemannian spatial metric:
\begin{equation*}
	h_{\fre^t} = \dd x \otimes \dd x + \mathfrak{L}_l(t+x ,y)^2 \dd y\otimes \dd y  + \mathfrak{L}_n(t+x , z)^2 \dd z \otimes \dd z \, ,
\end{equation*}

\noindent
on $\left\{t\right\}\times \mathbb{R}^3 \subset \cI\times \mathbb{R}^3$ is complete for all $t\in \cI$ we obtain a family of comoving globally hyperbolic metrics on $\cI\times \mathbb{R}^3$.  Implementing the change of coordinates:
\begin{equation*}
	x^{+} = \frac{t+x}{\sqrt{2}}\, , \qquad x^{-} = \frac{-t + x}{\sqrt{2}}\, ,
\end{equation*}

\noindent
the metric $g$ is given by:
\begin{equation*}
g =  \dd x^{+}\odot \dd x^{-} + \mathfrak{L}_l(x^{+} ,y)^2 \dd y\otimes \dd y  + \mathfrak{L}_n(x^{+} , z)^2 \dd z \otimes \dd z \, ,
\end{equation*}

\noindent
after a suitable redefinition of the functions $ \mathfrak{L}_l$ and $\mathfrak{L}_n$. This metric is a particular case of a Lorentzian metric expressed in Schimming coordiantes \cite{Schimming}, which exists in every Lorentzian manifold admitting a parallel null vector field. Equation \eqref{eq:ricciflatdiagonal} implies now that $g$ is Ricci-flat if and only if:
\begin{equation*}
\frac{\partial_{x^+} \partial_{x^+} \mathcal{L}_l }{\mathcal{L}_l}+\frac{\partial_{x^+} \partial_{x^+} \mathcal{L}_n }{\mathcal{L}_n}=0\, ,
\end{equation*}
 Some simple solutions, and thus Ricci-flat examples, can be found just by setting $\partial_{x^+} \partial_{x^+} \mathcal{L}_l=\partial_{x^+} \partial_{x^+} \mathcal{L}_n=0$:
\begin{equation*}
\mathcal{L}_l=w_1(y) x^++w_2(y)\, , \quad \mathcal{L}_n=w_3(z) x^++w_4(z)\,,
\end{equation*}
where $w_1,w_2,w_3,w_4$ are arbitrary smooth functions. 
\end{example}


\subsection{An example in Schimming coordinates.}


In reference \cite{Schimming} it was proven that any four-dimensional space-time $(M,g)$ equipped with a parallel light-like vector field $u^{\sharp} \in \mathfrak{X}(M)$ admits local coordinates $(x^+,x^-,y_1,y_2)$ in which the metric $g$ and the vector field $u^{\sharp}$ are written as follows:
\begin{equation*}
	g = \dd x^{+}\odot \dd x^{-} + k_{x^{+}} \, , \qquad u^{\sharp} = \frac{\partial}{\partial x^{-}}\, .
\end{equation*}

\noindent
where:
\begin{equation*}
	k_{x^{+}}(y_1,y_2) = k_{x^{+}\,ij}\dd y_i \otimes \dd y_j\, , \qquad i , j = 1,2\, .
\end{equation*}

\noindent
is a family of two-dimensional metrics parametrized by the coordinate $x^+$. A simple change of coordinates $\sqrt{2} x^{+} = t + x$ and $\sqrt{2} x^{-} = x-t$ allows to write the previous metric $g$ as:
\begin{equation}
	\label{eq:Schimming}
	g = -\dd t\otimes \dd t + \dd x\otimes \dd x + k_{t+x} \, ,  
\end{equation}

\noindent
whence we obtain a particular type of comoving globally hyperbolic spacetimes. Therefore, it is natural to study comoving parallel spinor flows adapted to the structure of  the metric \eqref{eq:Schimming}. Assume that the previous coordinate system is globally defined. Then, the Cauchy surface is given by $\Sigma = \mathbb{R}\times X$, with $X$ an oriented two-dimensional manifold, and the metric takes the form $h_t = \dd x\otimes \dd x + k_{t+x}$. Consequently, we assume that our comoving parallel spinor flow is of the form:
\begin{equation*}
	\fre^t = (\dd x , e^t_l(x) , e^t_n(x))\, ,
\end{equation*}

\noindent
where $k_{t+x} = e^t_l\otimes e^t_l + e^t_n\otimes e^t_n$. The comoving parallel spinor flow equations \eqref{eq:globhyperspinorcosmo} reduce to:
\begin{equation}
	\label{eq:comovingsimplified}
	\partial_t e^t_i(x) +  \Theta_t(e^t_i(x)) =  0\, , \quad \partial_x e^t_i(x) +  \Theta_t(e^t_i(x)) =  0\, , \quad \dd e^t_i(x)\vert_X  = 0 \, ,\quad  \Theta_t(\partial_x) = 0\, .
\end{equation}

\noindent
Hence, the comoving parallel spinor flow can be considered as a bi-parametric flow, parametrized by $t$ and $x$, for a family of closed oriented frames on $X$. In particular $(X,k_{t+x})$ is flat for every $(t,x)\in \mathbb{R}^2$ and therefore isometric to euclidean space, the flat cylinder or a flat torus. Equations \eqref{eq:comovingsimplified} immediately imply:
\begin{equation*}
	\partial_t e^t_i(x) = \partial_x e^t_i(x) \, , \quad i = l, n\, .
\end{equation*}

\noindent
Therefore, choosing coordinates $(y_1,y_2)$ on $X$, global for the plane and local for the torus and cylinder cases, every such family of solutions can be written as follows:
\begin{eqnarray*}
	e^t_l(x) = \mathfrak{f}^l_1(t+x) \dd y_1 + \mathfrak{f}^l_2(t+x) \dd y_2\, , \quad  e^t_n(x) = \mathfrak{f}^n_1(t+x) \dd y_1 + \mathfrak{f}^n_2(t+x) \dd y_2\, , 
\end{eqnarray*}

\noindent
for functions $\mathfrak{f}^l_1 , \mathfrak{f}^l_2 , \mathfrak{f}^n_1 , \mathfrak{f}^n_2 \in C^{\infty}(\mathbb{R})$ satisfying the following condition everywhere:
\begin{equation}
	\label{eq:constraintexample}
	\delta = \mathfrak{f}^l_1  \mathfrak{f}^n_2  - \mathfrak{f}^n_1 \mathfrak{f}^l_2 \neq 0 \, .
\end{equation}

\noindent
If this condition is satisfied, the dual frame is given by:
\begin{eqnarray*}
	e^t_l(x)^{\sharp} = \frac{1}{\delta}(\mathfrak{f}^n_2 \partial_{y_1} - \mathfrak{f}^n_1 \partial_{y_2})\, , \quad  e^t_n(x)^{\sharp} = \frac{1}{\delta}(- \mathfrak{f}^l_2 \partial_{y_1} + \mathfrak{f}^l_1 \partial_{y_2})\, . 
\end{eqnarray*}

\noindent
In order to guarantee that Equation \eqref{eq:constraintexample} is satisfied, we assume the following ansatz:
\begin{equation*}
	f^l_1 := e^{f_l}\,  \quad \mathfrak{p} := f^l_2 = -f^n_1\, , \quad f^n_2 := e^{f_n}\,
\end{equation*}

\noindent
in terms of functions $f_l\,f_n ,\mathfrak{p} \in C^{\infty}(\mathbb{R})$. This implies $\delta = e^{f_l+f_n} +\mathfrak{p}^2 > 0$ and therefore equations \eqref{eq:comovingsimplified} further reduce to: 
\begin{eqnarray}
\label{eq:comovingsimplifiedIIa}
& \partial_{x^{+}} e^{x^{+}}_l = (\partial_{x^{+}} e^{x^{+}}_l)((e^{x^{+}}_l)^{\sharp}) e^{x^{+}}_l + (\partial_{x^{+}} e^{x^{+}}_n)(( e^{x^{+}}_l)^{\sharp}) e^{x^{+}}_n\, ,\\
\label{eq:comovingsimplifiedIIb}
& \partial_{x^{+}} e^{x^{+}}_n = (\partial_{x^{+}} e^{x^{+}}_l)(( e^{x^{+}}_n)^{\sharp}) e^{x^{+}}_l + (\partial_{x^{+}} e^{x^{+}}_n)(( e^{x^{+}}_n)^{\sharp}) e^{x^{+}}_n\, ,
\end{eqnarray}

\noindent
where we have gone back to the coordinate $x^+ = t + x$ and written $e^{x^{+}}_i := e^t_i(x)$, $i = l , n$. In particular, note that the condition $(\partial_{x^{+}} e^{x^{+}}_n)(( e^{x^{+}}_l)^{\sharp})=(\partial_{x^{+}} e^{x^{+}}_l)(( e^{x^{+}}_n)^{\sharp})$ is necessarily satisfied, as required by Theorem \ref{thm:universalcover}. By direct computation one finds that Equations \eqref{eq:comovingsimplifiedIIa} and Equations \eqref{eq:comovingsimplifiedIIb} turn out to yield a single linearly independent equation which takes the form:
\begin{equation*}
\mathfrak{p}( \partial_{x^{+}}e^{f_l} + \partial_{x^{+}}e^{f_n}) =  (e^{f_l} +  e^{f_n})\partial_{x^{+}}\mathfrak{p}\, .
\end{equation*}
\noindent
The general solution of the previous equation is given by:
\begin{equation*}
	\mathfrak{p} = c\,(e^{f_l} +  e^{f_n})\, ,
\end{equation*}

\noindent
for any real constant $c$. Therefore we are led to the following Lorentzian metric, which by construction admits a parallel light-like vector field given by $\partial_{x^-}$:
\begin{equation}
\begin{split}
g=\dd x^+\odot \dd x^-+(e^{2f_l} +c^2 (e^{f_l}&+e^{f_n})^2 ) \dd y_1 \otimes \dd y_1 + c(e^{2f_l}-e^{2f_n}) \dd y_1 \odot \dd y_2 \\&+(e^{2f_n} +c^2 (e^{f_l}+e^{f_n})^2 ) \dd y_2 \otimes \dd y_2\,.
\end{split}
\label{eq:metricppwave}
\end{equation}

\noindent
The Ricci tensor of the previous metric is given by:
\begin{equation*}
\begin{split}
\mathrm{Ric}^g=\bigg [& 2 c^2 e^{f_l}((\partial_{x^+} f_l)^2+\partial_{x^+} \partial_{x^+} f_l)+2 c^2 e^{f_n}((\partial_{x^+} f_n)^2+\partial_{x^+} \partial_{x^+} f_n)\\&+(1+2c^2)e^{f_l+f_n}((\partial_{x^+} f_l)^2+\partial_{x^+} \partial_{x^+} f_l+(\partial_{x^+} f_n)^2+\partial_{x^+} \partial_{x^+} f_n ) \bigg ] \dd x^+ \otimes \dd x^+\, ,
\end{split}
\end{equation*}

\noindent
which vanishes if the following conditions are satisfied:
\begin{equation*}
(\partial_{x^+} f_l)^2+\partial_{x^+} \partial_{x^+} f_l=0 \, , \quad (\partial_{x^+} f_n)^2+\partial_{x^+} \partial_{x^+} f_n=0\,.
\end{equation*}
These ODEs are solved by:
\begin{equation*}
f_l(x^+)=a+\log\vert x^+-b\vert \, ,\quad f_n (x^+)=c+\log\vert x^+-d\vert \, ,
\end{equation*}
for real constants $a,b,c,d \in \mathbb{R}$. These solutions are well defined if $x^+ \in (-\infty, \mathrm{min}(b,d))$ or if $x^+ \in (\mathrm{max}(b,d),+\infty)$. However, although $f_l$ and $f_n$ present divergences whenever the argument of the logarithm vanishes, this is not problematic for the metric \eqref{eq:metricppwave} as long as $b \neq d$ (otherwise the metric would be degenerate at $x^+=b=d$), since both functions $f_l$ and $f_n$ appear exponentiated. In addition, it can be checked that the space time $(\mathbb{R}^2 \times X, g)$ is a plane wave, since the Riemann curvature tensor $\mathrm{R}^g:\Lambda^2 T (\mathbb{R}^2 \times X) \rightarrow \Lambda^2 T (\mathbb{R}^2 \times X)$ satisfies $\mathrm{R}^g \vert_{(\partial_{x^-})^\perp \wedge (\partial_{x^-})^\perp}=0$ and $\nabla_{V} \mathrm{R}^g=0$ for all $V \in (\partial_{x^-})^\perp$. 
 

\appendix


\phantomsection
\bibliographystyle{JHEP}


\end{document}